\documentclass[a4paper,10pt]{article}
\usepackage[utf8]{inputenc}

 \usepackage{tikz-cd} 
\usepackage{amsmath,amssymb,amsthm}
\usepackage[all]{xy}
\usepackage{graphicx}
\usepackage{color}
\usepackage[margin=1in]{geometry}
\usepackage[english]{babel} 

\usepackage{amsmath,amssymb}  
\usepackage{amsthm}           

\theoremstyle{plain}                          
 
\theoremstyle{definition}                     

\newtheorem{lemma}{Lemma}[section]
\newtheorem{remark}[lemma]{Remark}
\newtheorem{example}[lemma]{Example}
\newtheorem{theorem}[lemma]{Theorem}
\newtheorem{corollary}[lemma]{Corollary}
\newtheorem{definition}[lemma]{Definition}
\newtheorem{proposition}[lemma]{Proposition}

\newtheorem{question}[lemma]{Question}

\newtheorem{convention}[lemma]{Convention}

\newcommand{\field}[1]{\mathbb{#1}} 
\newcommand{\R}{\field{R}} 
\newcommand{\Z}{\field{Z}} 
\newcommand{\K}{\field{K}}

\newcommand{\Q}{\field{Q}}

\title{Hopf invariants and differential forms}
\author{Felix Wierstra}
\date{}

\begin{document}

\maketitle

\abstract{Let $f,g:M \rightarrow N$ be two maps between simply-connected smooth manifolds $M$ and $N$, such that $M$ is compact and $N$ is of finite $\R$-type. The goal of this paper is to use integration of certain differential forms to obtain a complete invariant of the real homotopy classes of the maps $f$ and $g$.}

\section{Introduction}

This paper is the sequel to \cite{Wie1}, in which we gave an answer to the following question.

\begin{question}
 Given two maps $f,g:M \rightarrow N$ between simply-connected smooth manifolds, such that $M$ is compact and $N$ is of finite $\R$-type. Can we find invariants of the maps $f$ and $g$, such that $f$ and $g$ are real homotopic if and only if these invariants agree?
\end{question}

In \cite{Wie1} we answered this question by defining a map 
$$mc_{\infty}:Map_*(M,N) \rightarrow \mathcal{MC}(M,N)$$
from the space of based maps from $M$ to $N$ to a certain moduli space of Maurer-Cartan elements in $Hom_{\R}(H_*(M;\R),\pi_*(N) \otimes \R)$, equipped with a certain $L_{\infty}$-algebra structure. This map is defined in two steps, first we define a map $mc:Map_*(M,N)\rightarrow Hom_{\R}(H_*(M;\R),\pi_*(N) \otimes \R)$ and then we take the quotient to the moduli space of Maurer-Cartan elements. 

The main problem with the results from \cite{Wie1} is that this invariant is in practice very hard to compute directly. The first main result of this paper is to make the map $mc$ computable. 

\begin{theorem}
Let $\{\varphi_{i,j}\}$ be a basis for $Hom_{\R}(H_*(M;\R),\pi_*(N) \otimes \R)$, the element $mc(f)$ can then be expressed in terms of this basis as   $mc(f)=\sum_{i,j}\lambda_{i,j}^f\varphi_{i,j}$. The coefficients $\lambda_{i,j}^f$ can be computed as certain integrals over certain subspaces of $M$.
\end{theorem}

The main significance of this theorem is that it makes it possible to obtain information about the real homotopy class of the map $f$ by computing a finite number of integrals. After having computed these integrals determining whether two Maurer-Cartan elements are equivalent or not becomes a completely algebraic problem in a finite dimensional $L_{\infty}$-algebra. So this theorem allows us to replace a very hard topological problem by a much easier algebraic problem.

The only problem so far is that we reduced the original problem to determining whether the Maurer-Cartan elements $mc(f)$ and $mc(g)$ are gauge equivalent or not. In practice this is often possible but can still be a very tedious problem. 

To solve this problem we will develop an algebraic analog of a CW-complex and an analog of the long exact sequence in homotopy associated to a fibration. These CW-complexes allow us to obtain a lot of information about the moduli space of Maurer-Cartan elements without doing any explicit computations. An example of one of the results  we can obtain this way is the following theorem.

\begin{theorem}
Let $f:M \rightarrow N$ be a map between simply-connected smooth manifolds, such that $M$ is compact and $N$ is of finite $\R$-type. The map $f$ is real homotopic to the constant map if and only if all the coefficients $\lambda_{i,j}^f$ are zero. 
\end{theorem}

In this paper we will work with the de Rham complex of differential forms and therefore we will only decide whether two maps are real homotopic. In Section \ref{secalternativeapproaches}, we will briefly explain how the ideas from this paper can be generalized to rational homotopy theory as well.

\subsection{Acknowledgments}

The author would like to thank Alexander Berglund for many useful conversations and ideas and for carefully reading earlier versions of this paper. The author also wishes to thank Dev Sinha for many useful conversations and ideas. Further the author would like to thank Bashar Saleh for carefully reading this paper and help with the examples. The author would also like to thank Joana Cirici and the anonymous referee for many comments on an earlier version of this paper. The author also acknowledges the financial support from Grant  GA CR  No. P201/12/G028.

\part{Preliminaries}

\section{Conventions}

We will now introduce the basic conventions we will use in this paper. In particular the conventions about real homotopy theory might not be very standard.

\subsection{Real homotopy theory}

In this paper we will, except for the last section, exclusively work with smooth manifolds. In this section we will recall some of the basic definitions and conventions.

\begin{convention}
In this paper we will assume that all the manifolds we consider are smooth and finite dimensional.
\end{convention}

\begin{definition}
Let $f,g:M \rightarrow N$ be two smooth maps between two simply-connected smooth manifolds $M$ and $N$. We call the maps $f$ and $g$ real homotopic if the induced maps $\Omega^{\bullet}(f),\Omega^{\bullet}(g):\Omega^{\bullet}(N) \rightarrow \Omega^{\bullet}(M)$, on the de Rham complexes are homotopic as maps of commutative differential graded algebras.
\end{definition}

\begin{remark}
For more details about real versus rational homotopy theory see for example \cite{DGMS1}.
\end{remark}

\subsection{Other conventions}

\begin{convention}
 In this paper $\K$ will denote a field of characteristic $0$ and will most of the time be the rationals $\Q$ or the reals $\R$.
\end{convention}

\begin{convention}
 In this paper we will assume that all the spaces we consider are simply-connected and all the CW-complexes are $1$-reduced, i.e. have only $1$ zero-cell and no one-cells.
\end{convention}

\begin{definition}
 The linear dual of a vector space $V$ will be denoted by $V^{\vee}$ and is defined as $Hom_{\K}(V,\K)$. 
\end{definition}

In this paper we will use the following definition of cohomotopy groups. 

\begin{definition}
 The rational cohomotopy groups of a space $X$ are defined as the linear dual of the rational homotopy groups and are denoted by $\pi^*(X)$, i.e. $\pi^*(X):=Hom_{\Z}(\pi_*(X),\Q)$.
\end{definition}

Note that this might differ from some definitions in the literature where the $n$th cohomotopy group of a space $X$ is defined as the set of maps $[X,S^n]$.

\begin{definition}\label{defdualhurewicz}
Let $X$ be a simply-connected space of finite $\Q$-type. Since we are working over a field $H^*(X)\cong H_*(X)^{\vee}$ and $\pi^*(X)\cong \pi_*(X)^{\vee}$. We therefore also get a canonical homomorphism $h^{\vee}:H^*(X) \rightarrow \pi^*(X)$ given by the dual of the Hurewicz homomorphism. 
\end{definition}

\begin{convention}\label{convallobjectsaredg}
In this paper we assume that all the operads, cooperads, algebras and coalgebras are differential graded. Further we will also assume that all operads and cooperads are reduced, i.e. $\mathcal{P}(0)=0$ and $\mathcal{P}(1)=\K$ for operads and $\mathcal{C}(0)=0$ and $\mathcal{C}(1)=\K$ for cooperads. We also assume that all the cooperads and coalgebras over cooperads we consider in this paper are conilpotent, see \cite{LV} for a definition.
\end{convention}

\begin{definition}
 We call a coalgebra $C$ simply-connected or $1$-reduced if $C_i=0$ for all $i \leq 1$.  
\end{definition}

\begin{convention}
 We will assume that all the spaces we consider are based and all the maps are based maps. Further we will assume that all the homology and cohomology is taken with real coefficients and reduced. We will further assume that all the homotopy and cohomotopy groups are tensored with $\R$ (or $\Q$ in the rational case).
\end{convention}


\begin{convention}
In this paper we will use a mixture between homological and cohomological gradings. Homology and homotopy are homologically  graded and cohomology and cohomotopy are cohomologically graded.
\end{convention}

\begin{convention}
 In this paper we work with shifted Lie and  $L_{\infty}$-algebras, i.e. all the brackets are of degree $-1$. This choice is mainly motivated by the fact that the Whitehead product on the homotopy groups of a space has degree $-1$. It will also turn out  to be more natural in some of the other constructions in this paper, see for example Theorem \ref{thrmLinftyconv}. We will denote the shifted $L_\infty$-operad by $s^{-1}L_\infty$ and the shifted Lie operad by $s\mathcal{LIE}$.
\end{convention}

\section{Algebras and operads}

In this paper we will make extensive use of the theory of operads, algebras over operads and twisting morphisms between them. We will almost always use the definitions, theorems and conventions of \cite{LV}, unless explicitly stated otherwise. The only difference is that we will denote the operadic bar construction on an operad $\mathcal{P}$ by $B_{op}\mathcal{P}$ and the operadic cobar construction on a cooperad $\mathcal{C}$ by $\Omega_{op}\mathcal{C}$. Let $\tau:\mathcal{C} \rightarrow \mathcal{P}$ be an operadic twisting morphism, for a $\mathcal{P}$-algebra $A$ and a $\mathcal{C}$-coalgebra $C$, the bar and cobar constructions relative to $\tau$ will be denoted by $B_{\tau}A$ and $\Omega_{\tau}C$. For the definitions of the relative bar and cobar constructions see Chapter 11 of \cite{LV}.  

There are two exceptions to this convention, for a commutative algebra $A$ we will denote the Lie coalgebraic bar construction by $B_{sLie}A$ and the bar construction relative to the twisting morphism $\pi:s^{-1}L_{\infty}^{\vee} \rightarrow \mathcal{COM}$ by $B_{s^{-1}L_{\infty}} A$. Note that we assumed that there is a shift in degree, because we work with shifted Lie and $L_\infty$-algebras. For the reader less familiar with operads and their algebras, for understanding the statements of this paper it is often enough to just use the $s\mathcal{Lie}$ and $s^{-1}L_{\infty}$-bar constructions.

\subsection{$\mathcal{P}_{\infty}$-algebras, $\mathcal{C}_{\infty}$-coalgebras and the Homotopy Transfer Theorem}

In this section we recall the definitions of algebras and coalgebras up to a sequence of coherent homotopies. Most of this section is based on Chapters 10 and 11 of \cite{LV}, but our definitions are slightly more general. The proofs are completely analogous.

There are several equivalent definitions of a $\mathcal{P}_{\infty}$-algebra (see Theorem 10.1.13 of \cite{LV}), in this paper we will only need two of them which we will describe now. 

\begin{definition}\label{defpinftydef1}
 Let $\mathcal{P}$ be an operad, a $\mathcal{P}_{\infty}$-algebra is defined as an algebra over $\Omega_{op}B_{op} \mathcal{P}$, the cobar-bar-resolution of $\mathcal{P}$.  
\end{definition}

\begin{definition}\label{defpinftydef2}
 A $\mathcal{P}_{\infty}$-structure on a vector space $A$ is a square-zero coderivation on the coalgebra $B_{op}\mathcal{P}(A)$.
\end{definition}

\begin{theorem}[\cite{LV} Theorem 10.1.13]\label{thrmrossetastone}
 Definitions \ref{defpinftydef1} and \ref{defpinftydef2} are equivalent.
\end{theorem}

\begin{remark}
Note that a coderivation on $C=\mathcal{C}(A)$, the cofree $\mathcal{C}$-coalgebra cogenerated by $A$, consists out of two parts. It is the sum $d_C=d_{\mathcal{C}}+d_{C}'$, of $d_{\mathcal{C}}$ the internal differential of $\mathcal{C}$ plus $d_{C}'$, a perturbation of the differential  $d_{\mathcal{C}}$. To go from this notion of $\mathcal{P}_{\infty}$-algebra to the definition of $\mathcal{P}_{\infty}$-algebra from Definition \ref{definfinitymorphism1} we take the image of the perturbation of the differential, i.e. the operation $\mu_c:A^{\otimes n} \rightarrow A$ is given by $\mu_c(a_1,...,a_n)=d_{c}'(c \otimes a_1 \otimes ... \otimes a_n)$, for $c \in \mathcal{C}(n)$ and $a_1,...,a_n \in A$.
\end{remark}

\begin{remark}
Note that these definitions of $\mathcal{P}_{\infty}$-algebras do not specialize to the classical notion of $\mathcal{A}_{\infty}$ and $s^{-1}L_{\infty}$-algebras. In this paper it will be necessary to also use $s^{-1}L_{\infty}$-algebras, these will be defined in Section \ref{seclinfty}. The reason we use this definition of $\mathcal{P}_{\infty}$-algebras is because it is always defined, and not just for Koszul operads.
\end{remark}

\subsection{$\infty$-morphisms}

Let $\mathcal{P}$ be an operad and let $A$ and $B$ be $\mathcal{P}$-algebras, let $\mathcal{C}$ be a cooperad and let $C$ and $D$ be $\mathcal{C}$-coalgebras. Let $\tau:\mathcal{C} \rightarrow \mathcal{P}$ be a Koszul operadic twisting morphism. In this section we will define $\infty_{\tau}$-morphisms between the algebras $A$ and $B$ and the coalgebras $C$ and $D$. The reason we consider $\infty_{\tau}$-morphisms is because they are morphisms up to a coherent sequence of homotopies and because the homotopy category of $\mathcal{P}$-algebras ($\mathcal{C}$-coalgebras) with $\infty_{\tau}$-morphisms is equivalent to the homotopy category of $\mathcal{P}$-algebras ($\mathcal{C}$-coalgebras). All the definitions and results are based on Chapters 10 and 11 of \cite{LV} and on the paper \cite{RNW1}, since we need slightly more general statements than in \cite{LV} we have formulated the definitions and theorems more generally. All the proofs are completely analogous to the proofs in \cite{LV}, see also \cite{RNW1}.

\begin{definition}
 An $\infty_{\tau}$-morphism between two $\mathcal{P}$-algebras $A$ and $B$ is a $\mathcal{C}$-coalgebra morphism $f:B_{\tau}A \rightarrow B_{\tau} B$. Because of Proposition 11.3.1 in \cite{LV} this map is completely determined by a linear map $B_{\tau} A \rightarrow B$, satisfying certain conditions. Since $B_{\tau}A$ is isomorphic as vector spaces to $\mathcal{C}\circ A$, this morphism breaks up in several components $f_n:\mathcal{C}(n) \otimes_{\Sigma_n} A^{\otimes n} \rightarrow B$. We will call this component the $n$th component of the $\infty_{\tau}$-morphism $f$. So equivalently an $\infty_{\tau}$-morphism of $\mathcal{P}$-algebras is a sequence of maps $f_n:\mathcal{C}(n) \otimes_{\Sigma_n} A^{\otimes n} \rightarrow B$, satisfying certain conditions. The category of $\mathcal{P}$-algebras with $\infty_{\tau}$-morphisms is denoted by $\infty_{\tau}$-$\mathcal{P}$-alg.
\end{definition}

\begin{definition}
Let $f:A \rightsquigarrow A'$ be an $\infty_{\tau}$-morphism of $\mathcal{P}$-algebras. We call $f$ an $\infty_{\tau}$-quasi-isomorphism if the arity $1$ component $f_1:A \rightarrow A'$ is a quasi-isomorphism.
\end{definition}

Dually we can also define $\infty_{\tau}$-morphisms between $\mathcal{C}$-coalgebras.

\begin{definition}
 An $\infty_{\tau}$-morphism between two $\mathcal{C}$-coalgebras $C$ and $D$ is a $\mathcal{P}$-algebra morphism $f:\Omega_{\tau} C \rightarrow \Omega_{\tau} D$. Because of Proposition 11.3.1 in \cite{LV}, this map is determined by a linear map $C \rightarrow \Omega_{\tau} D$, satisfying certain conditions. Since $\Omega_{\tau} D$ is isomorphic to $\mathcal{P} \circ D$, this morphism breaks up into several components $f_n:C \rightarrow \mathcal{P}(n) \otimes_{\Sigma_n} D^{\otimes n}$, we will call this  the $n$th-component of the $\infty_{\tau}$-morphism $f$. So equivalently an $\infty_{\tau}$-morphism of $\mathcal{C}$-coalgebra is equivalent to a sequence of maps $f_n:C \rightarrow \mathcal{P}(n) \otimes_{\Sigma_n} D^{\otimes n}$ satisfying certain conditions. The category of $\mathcal{C}$-coalgebra with $\infty_{\tau}$-morphism will be denoted by $\infty_{\tau}$-$\mathcal{C}$-coalg. 
\end{definition}

\begin{definition}
Let $g:C \rightsquigarrow C'$ be an $\infty_{\tau}$-morphism of $\mathcal{C}$-coalgebras. We call $g$ an $\infty_{\tau}$-quasi-isomorphism if the corresponding map $g:\Omega _{\tau} C \rightarrow \Omega_{\tau}C'$ is a quasi-isomorphism of $\mathcal{P}$-algebras and the arity $1$-component $g_1:C \rightarrow C'$ is a quasi-isomorphism.
\end{definition}

\begin{remark}
 Note that this definition of the $\infty_{\tau}$-morphism is dependent on the Koszul twisting morphism $\tau$, different choices for $\tau$ give different definitions of the $\infty_{\tau}$-morphisms. These definitions specialize to the definitions of \cite{LV}, whenever we assume that the operad $\mathcal{P}$ is binary quadratic Koszul and has a zero differential. The twisting morphism $\tau$ is then the canonical  twisting morphism from $\mathcal{P}^{\mbox{!`}}$ to $\mathcal{P}$. 
\end{remark}

\begin{convention}
To distinguish $\infty_{\tau}$-morphisms from strict morphisms we will denote $\infty_{\tau}$-morph\-isms by a $\rightsquigarrow$ arrow. Strict morphisms will be denoted by a normal arrow.
\end{convention}

The following theorem is a generalization of Theorem 11.4.8 in \cite{LV}. The proof is completely analogous and will therefore be omitted.

\begin{theorem}
Let $\tau:\mathcal{C} \rightarrow \mathcal{P}$ be a Koszul operadic twisting morphism. There is an equivalence of categories between the category of $\mathcal{P}$-algebras with $\infty_{\tau}$-morphisms and the homotopy category of $\mathcal{P}$-algebras with strict morphisms.
\end{theorem}

We have a similar statement for coalgebras which is as follows.

\begin{theorem}
 Let $\tau:\mathcal{C} \rightarrow \mathcal{P}$ be a Koszul operadic twisting morphism. There is an equivalence of categories between the  homotopy category of conilpotent $\mathcal{C}$-coalgebras with $\infty_{\tau}$-morphisms and the  homotopy category of conilpotent $\mathcal{C}$-coalgebras with strict morphisms. 
\end{theorem}

\subsection{The Homotopy Transfer Theorem}\label{secHTT}

Our main interest in $\mathcal{P}_{\infty}$-algebras and their $\infty_{\tau}$-morphisms is because of the Homotopy Transfer Theorem. This theorem allows us to transfer $\mathcal{P}_{\infty}$-algebra structures along contractions of chain complexes. This theorem and the explicit formulas we recall in this section will be important in Sections \ref{secalgebraichopfinv}, \ref{secfromspacestomc} and Part  \ref{partexample} where we will use them to obtain explicit formulas for the algebraic Hopf invariants. The formulas and theorems in this section come from \cite{Berg2} (see also \cite{LV} Chapter 10).

\begin{definition}\label{defcontraction2}
 A contraction of chain complexes is a diagram of the form
 $$\xymatrix{
V \ar@(ul,dl)[]|{h} \ar@/^/[rr]|p
&& W. \ar@/^/[ll]|{i} }$$
such that $deg(i)=deg(p)=0$ and $deg(h)=1$ and the maps $i$, $p$ and $h$ satisfy the following identities:
\begin{center}
$\partial(i)=0,$ $\partial(p)=0,$ $\partial(h)=ip-Id_V,$ \\
$pi=Id_W,$ $ph=0,$ $hh=0$ and $hi=0$.
\end{center}
\end{definition}

\begin{theorem}\label{thrmHTT}
Let $\mathcal{C}$ be a cooperad and denote by $\iota:\mathcal{C} \rightarrow \Omega_{op}\mathcal{C}$ the canonical operadic twisting morphism. Suppose that we have a contraction of chain complexes
$$\xymatrix{
A \ar@(ul,dl)[]|{h} \ar@/^/[rr]|p
&& B, \ar@/^/[ll]|{i} }$$
such that $A$ is an algebra over the operad $\Omega_{op}\mathcal{C}$.  Then there exists an $\Omega_{op} \mathcal{C}$-algebra structure on $B$, $\infty_{\iota}$-morphisms $P:A \rightsquigarrow B$ and $I:B \rightsquigarrow A$ and an $\Omega_{op}\mathcal{C}$-algebra contraction $H:A \rightsquigarrow A$, such that the maps $P$ and $I$ are $\infty_{\iota}$-quasi-isomorphisms and the diagram
$$\xymatrix{
A \ar@(ul,dl)[]|{H} \ar@/^/[rr]|P
&& B, \ar@/^/[ll]|{I} }$$
is a contraction of $\Omega_{op}\mathcal{C}$-algebras.
\end{theorem}


The rest of this section will be devoted to recall the explicit formulas for the transfered structure and transfered maps. These formulas will be important for obtaining explicit formulas for the Hopf invariants. We will use the formulas from \cite{Berg2}. 

First we will construct a contraction $H$. According to \cite{Berg2}  such a contraction is given by  

$$\mathbf{h}_n=\sum_{p+1+q=n}Id_A^{p}\otimes h \otimes (ip)^{q},$$
for non symmetric operads. For symmetric operad we need to symmetrize  and we get 

$$\mathbf{h}_n^{\Sigma}=\frac{1}{n !} \sum_{\sigma \in \Sigma_n} \sigma^{-1} \mathbf{h}_n \sigma.$$
Since we will only work with symmetric operads we will for simplicity denote $\mathbf{h}^{\Sigma}_n$ by $\mathbf{h}_n$. 

\begin{remark}
 There are some choices involved in defining the map $H$, in this paper we will fix these choices and always use the contraction $H$ from \cite{Berg2}. 
\end{remark}

Let $\nu \in \mathcal{C}(n)$, then we denote  the coproduct by 
$$\Delta(\nu)=\nu \circ 1^{n}+1 \circ \nu + \sum_{q=1}^{p}\nu^{q} \circ (\nu_1^q\otimes ... \otimes \nu^q_{r_q}) \sigma_q \in (\mathcal{C} \circ \mathcal{C})(n),$$
where $\nu^q$ and $\nu^q_i$ are elements of arity less than $n$ and $\sigma_q \in \Sigma_n$. The sum runs over all elements $\nu^q \in \mathcal{C}$ appearing in the coproduct, i.e. $p$ is the number of terms in the coproduct. The quadratic or infinitesimal part of the coproduct is defined as the part of the coproduct that is of the following form 
$$\Delta_{(1)}(\nu)=\sum_{i=1}^{u} (\nu'_i \circ_{e_i} \nu ''_i) \tau_i,$$
where $\tau_i\in \Sigma_n$ and the sum runs over all terms appearing in the infinitesimal coproduct, i.e. $u$ is the number of terms in the infinitesimal coproduct.

Further we will denote the $\Omega_{op}\mathcal{C}$-structure on $A$ by a sequence of maps $t^{\nu}:A^{\otimes n} \rightarrow A$, for $\nu \in \mathcal{C}(n)$. We will now give explicit formulas for the transferred structure and the transfer maps. Because of Theorem \ref{thrmrossetastone} the structure of a $\Omega_{op}\mathcal{C}$-algebra is the same as a square-zero derivation on $\mathcal{C}(A)$.     Since $\mathcal{C}(A)$ is cofree this derivation is completely determined by a map $t:\mathcal{C}(A) \rightarrow A$. We will denote by $(t')^{\nu}:A^{\otimes n} \rightarrow A$ the map defined by $(t')^{\nu}(a_1,...,a_n)=t'(\nu \otimes a_1 \otimes ... \otimes a_n)$, for $\nu \in \mathcal{C}$ and $a_1,...,a_n \in A$. Similarly the maps $P:B_{\iota}A \rightarrow B_{\iota}B$, $I:B_{\iota} B \rightarrow B_{\iota}A$ and $H:B_{\iota} A \rightarrow B_{\iota}A$ are also determined by their image on cogenerators and we denote the $\nu$-component of the maps $P$, $I$ and $H$ by $P^{\nu}$, $I^{\nu}$ and $H^{\nu}$.

\begin{theorem}
 The formulas for the transferred structure and the transferred maps $I$, $P$, $T$ and $H$ are given by:
$$(t')^{\nu}=pt^{\nu}i^{\otimes n}+\sum_{q=1}^{r} p t^{\nu^q}(i^{\nu_1^q}\otimes ... \otimes i^{\nu_{s_q}^q}) \sigma_q,$$
$$i^{\nu}=h t^{\nu} i^{\otimes n}+\sum_{q=1}^{r} h t^{\nu^q}(i^{\nu_1^q}\otimes ... \otimes i^{\nu_{s_q}^q}) \sigma_q,$$
$$p^{\nu}=(-1)^{\mid \nu \mid} p t^{\nu} \mathbf{h}_n+ \sum_{i=1}^{u} (-1)^{\mid \nu_i '' \mid}(p^{\nu_i '} \circ_{e_i} t^{ \nu_i ''}) \tau_i \mathbf{h}_n,$$
$$h^{\nu}=(-1)^{\mid \nu \mid} h t^{\nu} \mathbf{h}_n+ \sum_{i=1}^{u} (-1)^{\mid \nu_i '' \mid}(h^{\nu_i '} \circ_{e_i} t^{ \nu_i ''}) \tau_i \mathbf{h}_n.$$ 
In the first two formulas the sum runs over $\Delta(\nu)$, the coproduct of $\nu$ in the cooperad $\mathcal{C}$, i.e. $r$ is the number of terms appearing in the coproduct of $\nu$. In the third and fourth formula the sum runs over $\Delta_{(1)}(\nu)$, the infinitesimal coproduct of $\nu$, i.e. $u$ is the number of terms appearing in the infinitesimal coproduct of $\nu$.
\end{theorem}

For more details and proofs see \cite{Berg2}.

\subsection{$s^{-1}L_{\infty}$-algebras}\label{seclinfty}

In this section we will recall the basics about shifted $L_{\infty}$-algebras, one of the most important types of algebras in this paper. Most of this section is based on \cite{Getz1} and \cite{Berg1} in which proofs and details can be found, see also \cite{LV} Chapter 10.

\begin{definition}
 The cooperad $\mathcal{COCOM}$ is the cocommutative cooperad and is defined by  \\ $\mathcal{COCOM}(n)=\K$, for $n \geq 1$, with the trivial representation of $\Sigma_n$ in arity $n$. The decomposition map is given by 
 $$\Delta_{\mathcal{COCOM}}(\mu_n)=\sum_{p=1}^{n} \sum_{i=1}^{p} \mu_p \circ_i \mu_{n-p+1}.$$
 Where $\mu_n$ is the basis element of $\mathcal{COCOM}(n)$.
\end{definition}

\begin{definition}
 The $s^{-1}L_{\infty}$-operad is the operad defined as $\Omega_{op} \mathcal{COCOM}$, the cobar construction on the cocommutative cooperad.
\end{definition}

\begin{definition}
 An $s^{-1}L_{\infty}$-algebra $L$ is an algebra over the $s^{-1}L_{\infty}$-operad. In particular it is a graded vector space $L$ with a sequence of degree $-1$ operations $l_n:L^{\otimes n} \rightarrow L$ for each $n \geq 1$ satisfying certain conditions (see for example \cite{Getz1} or Section 10.1.12 in \cite{LV}). 
\end{definition}

\begin{remark}
 Note that the conditions in \cite{LV} and \cite{Getz1} are written down for $L_{\infty}$-algebras with operations of degree $2-n$, since our operations are shifted such that they all have degree $-1$ the signs will differ in most of the formulas.
\end{remark}

We will now make the notion of $\infty$-morphisms for $s^{-1}L_{\infty}$-algebras explicit.

\begin{definition}\label{definfinitymorphismforlielagebras}
 Let $L$ and $M$ be $s^{-1}L_{\infty}$-algebras. An $\infty$-morphism $\Phi:L \rightsquigarrow M$ of $s^{-1}L_{\infty}$-algebras is a map $\Phi:B_{\iota}L \rightsquigarrow B_{\iota}M$ between the bar constructions relative to the twisitng morphism $\iota:\mathcal{COCOM} \rightarrow s^{-1}L_{\infty}$. This is equivalent to a sequence of maps  $f_n:L^{\otimes n} \rightarrow M$ satisfying certain conditions (again see \cite{LV}).
\end{definition}

\begin{convention}
We will call the $\infty_{\iota}$-morphism from Definition \ref{definfinitymorphismforlielagebras}, just $\infty$-morphisms and omit the $\iota$.
\end{convention}

\begin{remark}
 In this definition of $\infty$-morphism we implicitly assumed that it is relative to the Koszul operadic twisting morphism $\iota:\mathcal{COCOM} \rightarrow s^{-1}L_{\infty}=\Omega_{op}\mathcal{COCOM}$.
\end{remark}

To make sure that certain sums converge, we need certain nilpotence conditions on the $s^{-1}L_\infty$-algebras we work with. We recall these conditions here, for more details see for example Section 2 of \cite{Berg1}.

\begin{definition}\label{deflowercentralseries}
Let $L$ be an $s^{-1}L_\infty$-algebra, then the lower central series of $L$ is defined as follows. Let $\Gamma^iL$ be subset of $L$ spanned by all possible bracket expressions one can form using at least $i$ elements of $L$. This defines a descending filtration
\[
L=\Gamma^1 L \supseteq \Gamma^2 L \supseteq ...,
\]
 which is called the lower central series. An $s^{-1}L_\infty$-algebra $L$ is called nilpotent if there exists an $N$ such that $\Gamma^nL=\{0\}$ for all $n>N$, $L$ is called degree-wise nilpotent if for each degree $d$ there exists an $N$ such that $(\Gamma^nL)_d=\{0\}$ for all $n>N$.  
\end{definition}

\begin{convention}
From now on we will assume that all $s^{-1}L_\infty$-algebras that model spaces, except for the mapping spaces, are nilpotent. As is shown in \cite{Berg1} this implies that the $L_\infty$-algebras modeling mapping spaces are degree-wise nilpotent.

\end{convention}

\begin{definition}
 Let $L$ be an $s^{-1}L_{\infty}$-algebra and $A$ be a CDGA. The extension of scalars of $L$ by $A$ is defined as the $s^{-1}L_{\infty}$-algebra whose underlying vector space is given by $A \otimes L$ and the bracket $l_n:( A\otimes L)^{\otimes n} \rightarrow A \otimes L$ is defined by $l_n(a_1\otimes x_1 ,...,a_n \otimes x_n)=(-1)^{\sum_{i<j} (\vert a_i \vert  \vert x_j \vert-1)}a_1...a_n \otimes l_n(x_1,...,x_n)$.
\end{definition}


\subsubsection{The Maurer-Cartan equation and twisted  $s^{-1}L_{\infty}$-algebras}

In an $s^{-1}L_{\infty}$-algebra we have a special class of elements called Maurer-Cartan elements, these elements are special since we can use them to twist the $s^{-1}L_{\infty}$-algebra $L$.

\begin{definition}
 Let $L$ be a degree-wise nilpotent $s^{-1}L_{\infty}$-algebra, a Maurer-Cartan element in $L$ is a degree $0$ element that satisfies the Maurer-Cartan equation which is given by
 $$\sum_{n\geq 1} \frac{1}{n!} l_n(\tau,...,\tau)=0.$$
 The set of solutions to the Maurer-Cartan equation will be denoted by $MC(L)$.
\end{definition}

In \cite{Getz1} Getzler associates to each nipotent $s^{-1}L_{\infty}$-algebra a simplicial set as follows. Let $\Omega_n$ be the algebra of polynomial de Rham forms on the $n$-simplex $\Delta^n$ (see \cite{Getz1} or \cite{FHT} for a precise definition). The simplicial CDGA $\Omega_{\bullet}$ is defined as the simplicial object in CDGA's which has $\Omega_n$ as its degree $n$ part. The face and degeneracy maps are induced by the cosimplicial structure of $\Delta_n$. Using $\Omega_{\bullet}$ we can define a functor from $s^{-1}L_{\infty}$-algebras to simplicial sets as follows. In \cite{Berg1}, Berglund extended Getzler's constructions to degree-wise nilpotent $s^{-1}L_\infty$-algebras.

\begin{definition}
 Let $L$ be a degree-wise nilpotent $s^{-1}L_{\infty}$-algebra, the nerve of $L$ is defined as the simplicial set 
 $$MC_{\bullet}(L)=MC(L \otimes \Omega_{\bullet}).$$
\end{definition}

In the rest of this section we will discus some of the properties of the functor $MC_{\bullet}$ and  explain how $MC_{\bullet}$ is used in rational homotopy theory. Since $MC_{\bullet}(L)$ is a simplicial set it gives us a notion of homotopy between the elements of $MC_0(L)$, in particular two elements $x,y \in MC_0(L)$ are called homotopy or gauge equivalent if there exists an element $z \in MC_1(L)$ such that $d_0(z)=x$ and $d_1(z)=y$. In \cite{Getz1}, it is shown that that this is an equivalence relation. From now on we will call two Maurer-Cartan elements which are homotopy equivalent in $MC_{\bullet}(L)$ gauge equivalent.
 
\begin{definition}
 The moduli space of Maurer-Cartan elements of a degree-wise nilpotent $s^{-1}L_{\infty}$-algebra $L$ is defined as the set of Maurer-Cartan elements of $L$ modulo the relation of gauge equivalence. The moduli space of Maurer-Cartan elements will be denoted by $\mathcal{MC}(L)$.
\end{definition}

One of the reasons we care about Maurer-Cartan elements is because we can use them to twist the $s^{-1}L_{\infty}$-structure on an $s^{-1}L_{\infty}$-algebra $L$.

\begin{definition}\label{deftwistedlinftyalg}
  Let $L$ be a degree-wise nilpotent $s^{-1}L_{\infty}$-algebra and let $\tau$ be a Maurer-Cartan element. The twisted $s^{-1}L_{\infty}$-algebra $L^{\tau}$ is defined as the $s^{-1}L_{\infty}$-algebra which has the same underlying vector space as $L$, but has twisted brackets and a twisted differential which are given by
  $$l_n^{\tau}(x_1,...,x_n)=\sum_{m \geq 0} \frac{1}{m!} l_{n+m}(\tau,...\tau,x_1,...,x_n),$$
  where the element $\tau $ appears $m$ times.
\end{definition}

See Proposition 4.4 of \cite{Getz1} for a proof that this gives indeed a new $L_{\infty}$-algebra structure on $L$.

\begin{lemma}[\cite{Berg1} Lemma 4.8]
 Let $L$ be a degree-wise nilpotent $s^{-1}L_{\infty}$-algebra and $\tau \in L$ a Maurer-Cartan element. The Maurer-Cartan set of $L^{\tau}$ is then described as follows
 $$MC(L^{\tau})=\{\sigma \in L_0 \mid \sigma + \tau \in MC(L)\}.$$
\end{lemma}

When we have an $\infty$-morphism of $s^{-1}L_{\infty}$-algebras, we also get a twist on this morphism. This will be important in Section \ref{secmodulispaces} where we will use this to study the moduli space of Maurer-Cartan elements. Proofs of these results can be found in Section 4 of \cite{Berg1}.

\begin{lemma}
 Let $f:L \rightsquigarrow M$ be an $\infty$-morphism between degree-wise nilpotent $s^{-1}L_{\infty}$-algebras $L$ and $M$. The map $f$ induces a map on Maurer-Cartan sets $f_*:MC(L) \rightarrow  MC(M)$, which is given by
 $$f_*(\tau)=\sum_{n \geq 1} \frac{1}{n!} f_n(\tau,...,\tau).$$
\end{lemma}

\begin{proposition}
Let $f:L \rightsquigarrow M$ be an $\infty$-morphism, and $\tau \in L$ be a Maurer-Cartan element. The $\infty$-morphism $f$ induces an $\infty$-morphism $f^{\tau}:L^{\tau} \rightsquigarrow M^{f_*(\tau)}$ which is given by 
$$f^{\tau}_n(x_1,...,x_n)=\sum_{l \geq 0} \frac{1}{l!} f_{n+l}(\tau,...,\tau,x_1,...,x_n).$$
The element $\tau$ appears $l$ times in the function $f_{n+l}(\tau,...,\tau,x_1,...,x_n)$.
\end{proposition}

\subsection{Convolution $s^{-1}L_{\infty}$-algebras}

In this section we recall Theorem 7.1 from \cite{Wie1} in which we define an $s^{-1}L_{\infty}$-algebra structure on the convolution algebra $Hom_{\K}(C,A)$ relative to an operadic twisting morphism.

\begin{theorem}\label{thrmLinftyconv}
 Let $\tau:\mathcal{C} \rightarrow \mathcal{P}$ be an operadic twisting morphism, let $C$ be a $\mathcal{C}$-coalgebra and let $A$ be a $\mathcal{P}$-algebra. There exists an $s^{-1}L_{\infty}$-algebra structure on the convolution algebra $Hom_{\K}(C,A)$, which is natural with respect to strict $\mathcal{C}$-coalgebra and strict $\mathcal{P}$-algebra morphisms. This $s^{-1}L_{\infty}$-structure has the following properties:
 \begin{itemize}
  \item The twisting morphisms relative to $\tau$ are the Maurer-Cartan elements in this $s^{-1}L_{\infty}$-algebra.
  \item Two algebra maps $f,g:\Omega_{\tau} C \rightarrow A$ are homotopic in the model category of $\mathcal{P}$-algebras if and only if the corresponding Maurer-Cartan elements $\tilde{f}$ and $\tilde{g}$ are gauge equivalent in the convolution $s^{-1}L_{\infty}$-algebra $Hom_{\K}(C,A)$.
 \end{itemize}
\end{theorem}

\begin{corollary}
 The set of homotopy classes of maps $[\Omega_{\tau}C,A]$ between $\Omega_{\tau}C$ and $A$, is in bijection with the moduli space of Maurer-Cartan elements $\mathcal{MC}(Hom_{\K}(C,A))$.
\end{corollary}

In the papers \cite{RN1} and \cite{RNW1} it is shown that  Theorem \ref{thrmLinftyconv} can be improved by showing that the $s^{-1}L_{\infty}$-convolution algebra is not only natural with respect to strict morphisms of $\mathcal{C}$-algebras and strict morphisms of $\mathcal{P}$-algebras, but also natural with respect to $\infty$-morphisms in one of the variables.  The following theorems are a summary of the main results of \cite{RNW1}.

\begin{definition}\label{definfinitymorphism1}
Let $\tau:\mathcal{C} \rightarrow  \mathcal{P}$ be a Koszul twisting morphism. Let $D$ be  a $\mathcal{C}$-coalgebra and let  $\phi:A \rightsquigarrow A'$ be an $\infty_{\tau}$-morphism  between $\mathcal{P}$-algebras $A$ and $A'$. Then we define an $\infty$-morphism $\phi_*:Hom_{\K}(D,A) \rightsquigarrow Hom_{\K}(D,A')$ of $s^{-1}L_{\infty}$-algebras by defining the arity $n$-part $(\phi_*)_n:Hom_{\K}(D,A)^{\otimes n}\rightarrow Hom_{\K}(D,A')$  of the induced $\infty$-morphism as: 
\begin{center}
	\begin{tikzpicture}
		\node (a) at (0,3){$D$};
		\node (b) at (2.5,3){$\mathcal{C}(D)$};
		\node (c) at (0,0){$A'$};
		\node (d) at (2.5,0){$\mathcal{C}(A)$,};
		\node (e) at (2.5,1.5){$\mathcal{C}(n)\otimes_{\Sigma_n}D^{\otimes n}$};
		
		\draw[->] (a) -- node[above]{$\Delta_D$} (b);
		\draw[->] (b) -- node[right]{$proj_n$} (e);
		\draw[->] (e) -- node[right]{$F$} (d);
		\draw[->] (d) -- node[above]{$\phi$} (c);
		\draw[dashed,->] (a) -- (c);
	\end{tikzpicture}
\end{center}
where $proj_n:\mathcal{C}(D)\rightarrow \mathcal{C}(n)\otimes_{\Sigma_n}D^{\otimes n}$ is the projection onto the arity $n$ part of $\mathcal{C}(D)$. The map $\Delta_D:D \rightarrow \mathcal{C}(D)$ is the decomposition map of $D$. Let $f_1\otimes ... \otimes f_n \in Hom_{\K}(D,A)^{\otimes n}$, then $F$ acts on $c \otimes x_1 \otimes... \otimes x_n \in \mathcal{C} \otimes D^{\otimes n}$ by 
$$F(c\otimes x_1 \otimes ... \otimes x_n)=\sum_{\sigma \in \Sigma_n}(-1)^{\theta} c \otimes  f_{\sigma(1)}(x_1) \otimes ... \otimes f_{\sigma(n)}(x_n)  .$$
The sign $\theta$ is given by
$$\theta = |p||F| + \sigma(F) + \sum_{i=1}^n|x_i|\left(\sum_{j=i+1}^n|f_{\sigma(j)}|\right) . $$
\end{definition}

Similarly we also define the induced $\infty$-morphism for $\infty_{\tau}$-morphisms of $\mathcal{C}$-coalgebras.

\begin{definition}\label{definfinitymorphism2}
Let $\tau:\mathcal{C}\rightarrow \mathcal{P}$ be a Koszul twisting morphism. Let $\psi:D'\rightsquigarrow D$ be an  $\infty_{\tau}$-morphism between $\mathcal{C}$-coalgebras $D '$ and $D$ and let $ A$ be a $\mathcal{P}$-algebra. Then we define an $\infty$-morphism $\psi^*:Hom_{\K}(D',A) \rightsquigarrow Hom_{\K}(D,A)$ of $s^{-1}L_{\infty}$-algebras by defining the arity $n$-part   $(\psi^*)_n:Hom_{\K}(D',A)^{\otimes n}\rightarrow Hom_{\K}(D,A)$  of the induced $\infty$-morphism as:  
\begin{center}
	\begin{tikzpicture}
		\node (a) at (0,3){$D'$};
		\node (b) at (2.5,3){$\mathcal{P}(D)$};
		\node (c) at (0,0){$A$};
		\node (d) at (2.5,0){$\mathcal{P}(A)$,};
		\node (e) at (2.5,1.5){$\mathcal{P}(n)\otimes_{\Sigma_n}D^{\otimes n}$};
		
		\draw[->] (a) -- node[above]{$\psi$} (b);
		\draw[->] (b) -- node[right]{$proj_n$} (e);
		\draw[->] (e) -- node[right]{$F$} (d);
		\draw[->] (d) -- node[above]{$\gamma_A$} (c);
		\draw[dashed,->] (a) -- (c);
	\end{tikzpicture}
\end{center}
where $proj_n:\mathcal{P}(D)\rightarrow \mathcal{P}(n)\otimes_{\Sigma_n}D^{\otimes n}$ is the projection onto the arity $n$ part of $\mathcal{P}(D)$, $\gamma_A:\mathcal{P}(A) \rightarrow A$ the composition map of $A$ and $F$ acts on $p\otimes x_1 \otimes ... \otimes x_n \in \mathcal{P} \otimes D^{\otimes n}$ by  
$$F(p\otimes x_1 \otimes ... \otimes x_n)=\sum_{\sigma \in S_n}(-1)^{\theta} p \otimes  f_{\sigma(1)}(x_1) \otimes ... \otimes f_{\sigma(n)}(x_n) .$$
The sign $\theta$ is given by
$$\theta = |p||F| + \sigma(F) + \sum_{i=1}^n|x_i|\left(\sum_{j=i+1}^n|f_{\sigma(j)}|\right) . $$
\end{definition}

\begin{theorem}[\cite{RNW1}, Corollary 5.4]\label{thrmbifunctorconvolutionalgebra}
Under the hypothesis of Definitions \ref{definfinitymorphism1} and \ref{definfinitymorphism2}, the bifunctor 
$$Hom(-,-): \mathcal{P}\mbox{-alg} \times (\mathcal{C}\mbox{-coalg})^{op}\rightarrow s^{-1}L_{\infty}\mbox{-alg}$$ 
extends to bifunctors 
$$Hom(-,-): \infty\mbox{-}\mathcal{P}\mbox{-alg} \times (\mathcal{C}\mbox{-coalg})^{op}\rightarrow s^{-1}L_{\infty}\mbox{-alg},$$
and 
$$Hom(-,-): \mathcal{P}\mbox{-alg} \times (\infty\mbox{-}\mathcal{C}\mbox{-coalg})^{op}\rightarrow s^{-1}L_{\infty}\mbox{-alg}.$$
So in particular the maps $\phi_*$ and $\psi^*$ from Definitions \ref{definfinitymorphism1} and \ref{definfinitymorphism2} are $\infty$-morphisms of $s^{-1}L_{\infty}$-algebras. 
\end{theorem}

\begin{remark}
As is shown in the example of Section 6 of \cite{RNW1} it is not possible to extend the bifunctor 
$$Hom(-,-): \mathcal{P}\mbox{-alg} \times (\mathcal{C}\mbox{-coalg})^{op}\rightarrow s^{-1}L_{\infty}\mbox{-alg}$$ 
to a bifunctor
$$Hom(-,-): \infty\mbox{-}\mathcal{P}\mbox{-alg} \times (\infty\mbox{-}\mathcal{C}\mbox{-coalg})^{op}\rightarrow s^{-1}L_{\infty}\mbox{-alg} .$$ 
\end{remark}

In this paper the following consequence of Theorem \ref{thrmbifunctorconvolutionalgebra} will be very important.

\begin{theorem}\label{thrmmcthingyorsomething?}
Let $\alpha:C_{\infty} \rightarrow s^{-1}L_{\infty}$ be a Koszul twisting morphism between the $C_{\infty}$-cooperad and the $s^{-1}L_{\infty}$-operad.
 \begin{enumerate}
  \item Let $\phi:D' \rightsquigarrow D$ be an $\infty_{\alpha}$-quasi-isomorphism between finite dimensional simply-connected $C_{\infty}$-coalgebras $D$ and $D'$. Let $L$ be a simply-connected $s^{-1}L_{\infty}$-algebra of finite type. The induced morphism $\phi^*:Hom_{\K}(D,L) \rightsquigarrow Hom_{\K}(D',L)$ then induces a weak equivalence on Maurer-Cartan simplicial sets
  $$MC_{\bullet}(\phi^*):MC_{\bullet}(Hom_{\K}(D ,L)) \rightarrow MC_{\bullet}(Hom_{\K}(D',L)).$$
  \item Let $\psi:L \rightsquigarrow L'$ be an $\infty_{\alpha}$-morphism between $L$ and $L'$, two simply-connected $s^{-1}L_{\infty}$-algebras of finite type. Let $D$ be a finite dimensional simply-connected $C_{\infty}$-coalgebra.  The induced morphism $\psi:Hom_{\K}(D,L) \rightsquigarrow Hom_{\K}(D,L')$ induces a weak equivalence of Maurer-Cartan simplicial sets
  $$MC_{\bullet}(\psi_*):MC_{\bullet}(Hom_{\K}(D,L)) \rightarrow MC_{\bullet}(Hom_{\K}(D,L')).$$  
 \end{enumerate}

\end{theorem}

\begin{proof}
 To prove the theorem we first filter $L$ and $L'$ by the filtration induced by the degree filtrations on $L$ and $L'$. Since we assumed that $L$ and $L'$ are simply-connected and of finite type, the filtration given by, $\mathcal{F}_n L$ is the ideal generated by $L_{\geq n+1}$, satisfies the conditions of Definition 8.1 of \cite{RNW1}. Proposition 8.9 of \cite{RNW1}, therefore applies and proves the theorem.
\end{proof}

\subsection{Rational models for spaces}\label{secratmodels}

Using the functor $MC_{\bullet}$ we can define rational models for spaces.

\begin{definition}
 A degree-wise nilpotent $s^{-1}L_{\infty}$-algebra $L$ is a rational model for a simply-connected space $X$ of finite $\Q$-type, if there exists a zig-zag of rational homotopy equivalences between $MC_{\bullet}(L)$ and $X$. 
\end{definition}

\begin{definition}\label{deflocallyfinite}
A degree-wise nilpotent $s^{-1}L_{\infty}$-model $L$ is called locally finite if every filtration quotient $L / \Gamma_nL$ is finite dimensional. Where $\Gamma_nL$ is the lower central series of $L$ (see Definition \ref{deflowercentralseries}). 
\end{definition}

There are also $C_{\infty}$-coalgebra models for spaces, where $C_{\infty}$ is the cooperad defined by $C_{\infty}=B_{op} \mathcal{s^{-1}\mathcal{LIE}}$. To define these we will first recall one of  the main results of \cite{Quil1}.

\begin{theorem}\label{thrmquillen}
 There exists a functor $\mathcal{C} \lambda: Top_{*,1} \rightarrow CDGC_{\geq 2}$, from the category of simply-connected based topological spaces  with rational equivalences to the category of  simply-connected cocommutative coalgebras, such that $\mathcal{C} \lambda$ induces an equivalence between the homotopy categories.
\end{theorem}

\begin{definition}
 A $C_{\infty}$-coalgebra $C$ is a $C_{\infty}$-model for a simply-connected space $X$, if $C$ can be connected by a zig-zag of  quasi-isomorphisms to $\mathcal{C} \lambda(X)$.
\end{definition}

\begin{definition}
A $C_{\infty}$-model for a simply-connected manifold $X$ of finite $\Q$-type is a $C_{\infty}$-algebra $A$, such that $A$ can be connected to $\Omega^{\bullet}(M)$ by a zig-zag of quasi-isomorphisms.
\end{definition}

\begin{remark}
It is possible to replace the de Rham complex $\Omega^{\bullet}$ by the complex of polynomial de Rham forms to generalize this definition to all simply-connected spaces of finite $\Q$-type. See Chapter 10 of \cite{FHT} for more details. 
\end{remark}

\subsection{The relation between homology and the homotopy groups}

The relation between $C_{\infty}$ and $s^{-1}L_{\infty}$-models is given by the bar and cobar construction. In this section we will briefly describe how this relation works. The results of this section are generalizations of the ideas of  Chapter 22 of \cite{FHT}.

\begin{theorem}\label{thrmquillenscanlfunctors}
Let $\tau:C_{\infty} \rightarrow s^{-1}L_{\infty}$ be a Koszul twisting morphism and $C$ a simply-connected $C_{\infty}$-coalgebra model for a simply-connected space $X$. Then $\Omega_{\tau}C$ is a degree-wise nilpotent $s^{-1}L_{\infty}$-model for $X$ and in particular $H_*(\Omega_{\tau}C)$ is isomorphic as graded vector spaces to $\pi_*(X)$. 

Dually, let $L$ be a simply-connected degree-wise nilpotent $s^{-1}L_{\infty}$-model of finite type for $X$, then $B_{\tau}L$ is a $C_{\infty}$-coalgebra model for $X$.
\end{theorem}

The following theorem states the dual result of Theorem \ref{thrmquillenscanlfunctors}. The following theorem is a variation of Theorem 2.10 of \cite{SW2}.

\begin{theorem}
Let $\kappa:s^{-1}L_{\infty}^{\vee} \rightarrow C_{\infty}$ be a Koszul twisting morphism and $A$ a $C_{\infty}$-algebra model for a simply-connected space $X$ of finite $\Q$-type. Then there is an isomorphism of graded vector spaces between $H^*(B_{\tau}A)$ and $\pi^*(X)=Hom_{\Z}(\pi_*(X),\R)$.
\end{theorem}

\begin{convention}\label{convcohomotopygroups}
From now on we will by abuse of notation sometimes denote the homology of the cobar construction on a coalgebra $C$ by $\pi_*(C)$, i.e. $\pi_*(C)=H_*(\Omega_{\tau}C)$. Similarly we will denote the cohomology of the bar construction of a $C_{\infty}$-algebra $A$ by $\pi^*(A)$, i.e. $\pi^*(A)=H^*(B_{\tau}A)$.
\end{convention}

\subsection{The model categories $\mathcal{C}$-coalgebras and  $\mathcal{P}$-algebras}

In this subsection we recall the model structures on the categories of $\mathcal{P}$-algebras and $\mathcal{C}$-coalgebras. These model structures are important in the rest of this paper because they give us a good framework for doing homotopy theory. In this section we will assume that $\mathcal{P}$ is an operad, $\mathcal{C}$ a cooperad and $\tau:\mathcal{C} \rightarrow \mathcal{P}$ a Koszul operadic twisting morphism. 

\begin{theorem}[\cite{Hin1} Theorem 4.1.1]
 The category of $\mathcal{P}$-algebras has a model structure in which the weak equivalences are given by quasi-isomorphisms, the fibrations are given by surjective maps and the cofibrations are the maps with the left lifting property with respect to acyclic fibrations.
\end{theorem}

\begin{theorem}[\cite{Val1} Theorem 2.1]
The category of $\mathcal{C}$-coalgebras has a model structure in which the weak equivalences are the maps $f:X \rightarrow Y$, such that $\Omega_{\tau}(f):\Omega_{\tau} X \rightarrow  \Omega_{\tau}Y$ is a weak equivalence of $\mathcal{P}$-algebras.  The cofibrations are the degree-wise monomorphisms and the fibrations are the maps with the right lifting property with respect to acyclic cofibrations.
\end{theorem}

From now on we will always assume that we are working in one of these model categories.

\subsection{Rational models for mapping spaces}

In this section we  recall how to construct rational models for mapping spaces using Theorem \ref{thrmLinftyconv}. To construct this model we let $\tau:C_{\infty} \rightarrow s^{-1}L_{\infty}$ be a Koszul operadic twisting morphism from the $C_{\infty}$-cooperad to the $s^{-1}L_{\infty}$-operad (see \cite{Wie1} Section 11 for a construction of such a twisting morphism). 

Recall that an $s^{-1}L_{\infty}$-algebra is called locally finite if every filtration quotient $L / \Gamma_nL$ of the lower central series is finite dimensional, (see Definition \ref{deflocallyfinite}). 

\begin{theorem}[\cite{RNW1}, Corollary 9.20]
 Let $C$ be a $C_{\infty}$-model for a finite simply-connected CW-complex  $X$ and let $L$ be a simply-connected degree-wise nilpotent locally finite $s^{-1}L_{\infty}$-model for a simply-connected rational space $Y_{\Q}$, such that $Y_{\Q}$ is of finite $\Q$-type. The convolution algebra $Hom_{\Q}(C,L)$ equipped with the $s^{-1}L_{\infty}$-structure from Theorem \ref{thrmLinftyconv} is a rational model for the mapping space $Map_*(X,Y_{\Q})$, i.e. we have a homotopy equivalence
 $$Map_*(X,Y_{\Q}) \simeq MC_{\bullet}(Hom_{\Q}(C,L)).$$
\end{theorem}

\subsubsection{Change of base point}

Since a mapping space often consists of several connected components it is important to know how to change the base point from one connected component to another. In \cite{Berg1} Theorem 1.2, this was done and we will recall this theorem here.

\begin{theorem}\label{thrmchangeofbasepoint}
 If $L$ is a degree-wise nilpotent $s^{-1}L_{\infty}$-algebra, then there exists a natural group isomorphism
 $$B:H_n(L^{\tau}) \rightarrow \pi_{n}(MC_{\bullet}(L),\tau),$$
 for every $\tau\in MC_0(L)$.
\end{theorem}

This theorem tells us that when we are considering mapping spaces we can compute the rational homotopy groups of the component of $\tau$ by simply twisting the $L_{\infty}$-algebra $L$ by the element $\tau$. 

\section{Algebraic CW-complexes and the long exact sequence in homotopy for coalgebras}\label{secalgebraicCWcomplex}

A classical theorem states that if $f:E \rightarrow B$ is a fibration of topological spaces with fiber $F$, then this induces a long exact sequence on the level of  homotopy groups.   In this section we will use this long exact sequence to define an analogous long exact sequence associated to certain maps of algebras and coalgebras. In this section we will assume that all algebras and coalgebras are defined over a field $\K$ of characteristic $0$. To do this we will first recall the definition of the long exact sequence associated to a fibration of spaces. Then we will define algebraic CW-complexes, which are the algebraic analog of CW-complexes of topological spaces. After that we will show that to each fibration of algebras we can associate a long exact sequence for the homotopy groups of that algebra. The long exact sequence and the algebraic CW-complexes will be important in Section \ref{secmodulispaces}, where they will help us to understand the moduli space of Maurer-Cartan elements.

\subsection{The long exact sequence associated to a fibration of spaces}\label{seclongexactsequence}

In this section we will recall the long exact sequence for the homotopy groups associated to a fibration. This exact sequence  will be important to obtain information about the moduli space of Maurer-Cartan elements in Section \ref{secmodulispaces}. A proof for Theorem \ref{thrmlesfibrations} and Lemma \ref{lemlesfibration} can be found in Chapter 1 of \cite{MP11} on which this section is based.

\begin{theorem}\label{thrmlesfibrations}
 Let $F \rightarrow E \rightarrow B$ be a fibration of spaces or simplicial sets with fiber $F$. There is a long exact sequence in homotopy groups given by 
 $$... \rightarrow \pi_{n+1}(B,b_B) \rightarrow \pi_n(F,b_F) \rightarrow \pi_n(E,b_E) \rightarrow \pi_n(B,b_B) \rightarrow ... .$$
 Where $b_B$ is the base point of the space $B$.
\end{theorem}

We would like to apply this exact sequence to the attaching maps in a CW-complex. The following lemma states some basic facts about cell attachments of CW-complexes. Its proof will be omitted.

\begin{lemma}\label{lemthecofibersequence}
 Let $a_n:\bigvee_{k_{n}}S^{n-1} \rightarrow K_{n-1}$ be the attaching map of the $n$-cells to the $(n-1)$-skeleton $K_{n-1}$ of a CW-complex $K$ and let $Y$ be a space. Then we get a cofiber sequence
 $$\bigvee_{k_{n}}S^{n-1} \rightarrow K_{n-1} \rightarrow K_n ,$$
 where $K_n$ is the cone of the attaching map $a_n$, i.e. $K_n$ is $K_{n-1}$ with the $k_n$ cells attached along the map $a_n$. We will denote the inclusion of $K_{n-1}$ into $K_n$ by $i_n:K_{n-1} \rightarrow K_n$. When we apply the functor $Map_*(-,Y)$ we get a fiber  sequence
 $$Map_*(K_{n},Y) \xrightarrow{i_n^*} Map_*(K_{n-1},Y) \xrightarrow{a_n^*} Map_*(\bigvee_{k_n}S^{n-1},Y).$$
\end{lemma}

In this paper we are mainly interested in the last part of  the long  exact sequence, associated to this fiber sequence. This part of the exact sequence is no longer an exact sequence of groups but an exact sequence of pointed sets. Since we will only apply this to cell attachment maps, we get the following lemma (see Lemma 1.4.6 in \cite{MP11}).

\begin{lemma}\label{lemlesfibration}
The following statements hold about the  $\pi_1$ and $\pi_0$ part of the long exact sequence of homotopy groups associated to the fiber sequence 
$$Map_*(K_{n},Y) \xrightarrow{i_n ^*} Map_*(K_{n-1},Y) \xrightarrow{a_n^*} Map_*(\bigvee_{k_n}S^{n-1},Y),$$ 
from Lemma \ref{lemthecofibersequence}. The long exact sequence is given by
$$... \xrightarrow{\pi_1(i_n^*)} \pi_1(Map_*(K_{n-1},Y),b_1) \xrightarrow{\pi_1(a_n^*)} \pi_1(Map_*(\bigvee_{k_n}S^{n-1},Y),b_2) $$
$$\xrightarrow{\partial} [K_n,Y] \xrightarrow{i_n^*} [K_{n-1},Y]\xrightarrow{a_n^*)}[\bigvee_{k_n}S^{n-1},Y]  ,$$
where $b_1$ and $b_2$ are the base points of $Map_*(K_{n-1},Y)$ and \\ $Map_*(\bigvee_{k_n}S^{n-1},Y)$. Note that $\pi_1(Map_*(\bigvee_{k_n}S^{n-1},Y),b_2)$ is isomorphic to $\bigoplus_{k_n}\pi_n(Y)$.
 \begin{enumerate}
  \item The group $\pi_1(Map_*(\bigvee_{k_{n}} S^{n-1},Y),b_2)$ acts from the right on $[K_{n},Y]$, the set homotopy classes of based maps between $K_{n}$ and $Y$.
  \item The map $\pi_1(Map_*(\bigvee_{k_{n}} S^{n-1},Y),b_2)\xrightarrow{i^*_n} [K_{n},Y]$ is a map of right \\  $\pi_1(Map_*(\bigvee_{k_{n-1}} S^{n-1},Y),b_2)$-sets.
  \item Let $x,y \in [K_n,Y]$, then we have $i_n^*(y)=i_n^*(x)$ if and only if there exists an element $z \in \pi_1(Map_*(\bigvee_{k_n}S^{n-1},Y),b_2)$ such that $y=z \cdot x$, where  $z$ acts on $x$ according to part $(1)$ of this lemma.
  \item Denote by $\partial:\pi_1(Map_*(\bigvee_{k_{n}} S^{n-1},Y),b_2)\rightarrow [K_n,Y]$ the connecting homomorphism. Let $x,y \in \pi_1(Map_*(\bigvee_{k_{n}} S^{n-1},Y),b_2)$  then  $\partial(x)=\partial(y)$ if and only if $y=\pi_1(a_n^*)(z) \cdot x $ for some $z \in \pi_1(Map_*(K_{n-1},Y))$ and $\pi_1(a_n^*)$ is the induced map of fundamental groups $\pi_1(a_n^*):\pi_1(Map_*(K_{n-1},Y)) \rightarrow \pi_1(Map_*(\bigvee_{k_n}S^{n-1},Y))$. 
 \end{enumerate}
\end{lemma}

\subsection{Algebraic CW-complexes}

In this section we will describe the algebraic analog of a CW-complex. These algebraic CW-complexes will be important in Section \ref{secmodulispaces} where we will give a description of the moduli space of Maurer-Cartan elements in terms of the attaching maps of the algebraic CW-complex.  An algebraic CW-complex is a coalgebra   inductively build  out of cells, similar to Sullivan algebras. In this section we assume that we work with coalgebras over a cooperad $\mathcal{C}$ of the form $\mathcal{C}=B_{op}\mathcal{P}$ for some operad $\mathcal{P}$. We denote by $\pi:\mathcal{C} \rightarrow \mathcal{P}$ the canonical  operadic Koszul twisting morphism from $\mathcal{C}$ to $\mathcal{P}$. 

Similar to spaces we would like to build coalgebras out of push-outs of the algebraic equivalent of a disc. 

\begin{definition}
 
 The disk coalgebra $\mathfrak{D}^n$ is defined as the differential graded vector space with one generator $\alpha$ in degree $n-1$ and one generator $\beta$ in degree $n$, the differential is given by $d(\beta)=\alpha$. We will equip $\mathfrak{D}^n$ with the trivial coalgebra structure. The sphere coalgebra $\mathfrak{S}^n$ is defined as the one dimensional coalgebra with one basis element $\alpha$ in degree $n$, $\mathfrak{S}^n$ is equipped with the trivial coalgebra structure. The inclusion of the sphere coalgebra $\mathfrak{S}^{n-1}$ into $\mathfrak{D}^n$ will be denoted by $i:\mathfrak{S}^{n-1} \rightarrow \mathfrak{D}^n$. 
\end{definition}

A cell attachment is now defined as follows.

\begin{definition}
 Let $C$ be a $\mathcal{C}$-coalgebra, $\mathfrak{D}^n$ a disk coalgebra and $f:\mathfrak{S}^{n-1} \rightsquigarrow C$ an $\infty_{\pi}$-morphism of $\mathcal{C}$-coalgebras, i.e. a strict morphism $f:\Omega_{\pi}\mathfrak{S}^{n-1} \rightarrow \Omega_{\pi} C$ of $\mathcal{P}$-algebras. The cell attachment of $\mathfrak{D}^n \cup_f C$ is then defined as the $\mathcal{C}$-coalgebra corresponding to the following push-out of $\mathcal{P}$-algebras. 
 $$
 \xymatrix{
 \Omega_{\pi} \mathfrak{S}^{n-1} \ar[d]^f \ar[r]^i & \Omega_{\pi} \mathfrak{D}^n  \ar[d] \\
 \Omega_{\pi} C \ar[r] &   \Omega_{\pi} \mathfrak{D}^n \cup_f C
 }
 $$
\end{definition}

Since $\Omega_{\pi } \mathfrak{S}^{n-1}$, $\Omega_{\pi} \mathfrak{D}^n$ and $\Omega_{\pi} C$ are all free, the push out is given by the free $\mathcal{P}$-algebra generated by the vector space $C \oplus \mathfrak{S}^n$, with some differential $d$. Because of Definition \ref{defpinftydef2} and Theorem \ref{thrmrossetastone},  a differential $d$ on a free $\mathcal{P}$-algebra is equivalent to a $\mathcal{C}$-coalgebra structure on the vector space $C \oplus \mathfrak{S}^n$. This push-out therefore defines a well defined $\mathcal{C}$-coalgebra structure on $C \oplus \mathfrak{S}^n$.

\begin{definition}
 An algebraic CW-complex $C$ is a coalgebra $C$ which is zero in degree less or equal than $0$ and is inductively built up out of cells of increasing dimension. In particular we start with the $1$-skeleton which is equal to the direct sum of $k_1$-copies of $\mathfrak{S}^1$. Then we attach $2$-cells to obtain the $2$-skeleton and proceed inductively by attaching $n$-cells to the $(n-1)$-skeleton. It is not allowed to attach $m$-cells to the $n$-skeleton when $m \leq n$. 
\end{definition}

\begin{remark}
 In theory we allow infinitely many cell in our algebraic CW-complexes, but for all our applications we will always assume that our algebraic CW-complexes have finitely many cells.
\end{remark}

\begin{definition}
 We call an algebraic CW-complex $C$, $r$-connected or $r$-reduced if $C_{\leq r}=0$, i.e. it has no cells in dimension $r$ or below.
\end{definition}

We will now give an example of an algebraic CW-complex for the cellular chains of $S^2 \times S^2$. More details about this example can be found in Proposition \ref{propsnxsmcw}. Recall that we assumed that all our homology is reduced, we will therefore always ignore the $0$-cell of the algebraic CW-complexes.

\begin{example}\label{Exs2xs2cw}
 The space $S^2 \times S^2$ has a CW-decomposition in spaces defined by taking two $2$-cells with one $4$-cell attached via the Whitehead product of the $2$-cells. The homology coalgebra of this space has as underlying vector space $\Q\alpha \oplus \Q \beta \oplus \Q \gamma$ with $\vert\alpha\vert=\vert \beta\vert=2$ and $\vert\gamma\vert=4$ and the coproduct is given by $\Delta (\alpha)=\Delta( \beta) =0$ and $\Delta(\gamma)=\alpha \otimes \beta +\beta \otimes \alpha$. An algebraic CW-complex, as a $C_{\infty}$-coalgebra, for this space is given by attaching the $4$-cell $\gamma$ to $\Q \alpha \oplus \Q \beta$ via the $\infty_{\pi}$-morphism $\upsilon:\mathfrak{S}^3 \rightsquigarrow \mathfrak{S}^2 \oplus \mathfrak{S}^2$, given by  $\upsilon_2(\epsilon)=[\alpha,\beta]$ and zero otherwise, where $\epsilon$ is the boundary of $\gamma$.  
\end{example}

In the rest of this section we will prove some results about the existence of algebraic CW-complexes. To do this we need some mild restrictions on our operads. The following definition is Definition 4.1 in \cite{CR1}.

\begin{definition}
 An operad $\mathcal{P}$ is called $r$-tame if $\mathcal{P}(n)_q=0$ for all $q \leq (1-n)(1+r)$.
\end{definition}

\begin{theorem}\label{propCWdecomp}
 Let $C$ be a finite dimensional $r$-connected $\mathcal{C}$-coalgebra over a fibrant cooperad $\mathcal{C}$ of the form $\mathcal{C}= B_{op} \mathcal{P}$, such that $\mathcal{P}$ is $r$-tame, then $C$ has a CW-decomposition. 
\end{theorem}

\begin{remark}
 It is easy to see that the condition that the cooperad $\mathcal{C}$ is zero in degree less than $-n+1$ is satisfied for all the cooperads that are important in this paper. In particular the cooperads $s^{-1}\mathcal{LIE}^{\vee}$, $s^{-1}L_{\infty}^{\vee}$, $B_{op} \Omega_{op} \mathcal{COCOM}$ and $C_{\infty}^{\vee}$ satisfy this condition.
\end{remark}

\begin{proof}
 To define an algebraic CW-complex structure for the coalgebra $C$ we will use the degree filtration on $C$. It follows from Theorem 4.6 in \cite{CR1} and our assumptions that the $\mathcal{P}$-algebra $\Omega_{\tau} C$ can be built up out of a sequence of cell attachments. Therefore the coalgebra $C$ can also be built up out of a sequence of cell attachments.
\end{proof}

\begin{theorem}\label{thrmexistenceofalgebraiccwcomplexes}
Let $X$ be a $1$-reduced CW-complex of finite type, i.e. $X$ has one $0$-cell and no $1$-cells and only finitely many cells in every degree. Then there exists an algebraic CW-complex $C$, with exactly one basis element for each cell of $X$, such that $C$ is a $C_{\infty}$-model for $X$.
\end{theorem}

\begin{proof}
To prove the theorem we will use Section (e) of Chapter 24 of \cite{FHT}, which constructs a Lie model for the CW-complex $X$, with one generator for each cell. Since this shifted Lie model is quasi-free, it corresponds to a $C_{\infty}$-structure on the set of generators. We can now apply Theorem \ref{propCWdecomp} to the generators of this Lie model to get a CW-decomposition for $X$ with exactly one cell for each cell of $X$. 
\end{proof}

\begin{lemma}
 Let $C$ be an algebraic CW complex for some cooperad $\mathcal{C}$ and denote by $C_{\leq n}$ the $n$-skeleton, the inclusion maps $j_n:C^{ \leq n} \rightarrow C^{\leq n+1}$ are $\mathcal{C}$-coalgebra homomorphisms.
\end{lemma}

\begin{proof}
 This is a straightforward check which is left to the reader.
\end{proof}

\subsection{The long exact sequence}

We will now prove a theorem which is the algebraic analog of Theorem \ref{thrmlesfibrations} and Lemma \ref{lemlesfibration}, which associates a long exact sequence to each map of coalgebras.

\begin{theorem}\label{thrmlongexactsequencealgebraiccwcomplexes}
  Let $D$ be a finite dimensional $C_{\infty}$-coalgebra and let  $f:\mathfrak{S}^n \rightsquigarrow D$ be an attaching map. Denote by $C(f)$ the coalgebra obtained by attaching an $(n+1)$-cell to $D$ via the map $f$. Let $\tau:C_{\infty} \rightarrow s^{-1}L_{\infty}$ be a Koszul twisting morphism from the $C_{\infty}$-cooperad to the $s^{-1}L_{\infty}$-operad. Let $L$ be a simply-connected degree-wise nilpotent $s^{-1}L_{\infty}$-algebra of finite type, assume that we equip the spaces $Hom_{\K}(C,L)$ with the $s^{-1}L_{\infty}$-structure from Theorem \ref{thrmLinftyconv}. Then we get a homotopy  fibration sequence of simplicial sets given by
  $$MC_{\bullet}(Hom_{\K}(C(f),L)) \rightarrow MC_{\bullet}(Hom_{\K}(D,L)) \rightarrow MC_{\bullet}(Hom_{\K}(\mathfrak{S}^n,L)).$$
  In particular it induces a long exact sequence in homotopy as in Theorem \ref{thrmlesfibrations} and Lemma \ref{lemlesfibration}.
\end{theorem}

\begin{remark}
The existence of the Koszul twisting morphism $\tau:C_{\infty} \rightarrow s^{-1}L_{\infty}$ is shown in Section 11 of \cite{Wie1}.
\end{remark}

Before we prove the theorem we will first recall a theorem by Hinich  which states that the category of Lie algebras is a simplicial model category. This theorem can be found as Theorem 2.4 in \cite{Hin2}.

\begin{theorem}\label{thrmHinichsresult}
 The category of shifted Lie algebras admits a simplicial model structure, in which the fibrations are given by the surjective maps, the weak equivalences by quasi-isomorphisms and the cofibrations by maps with the left lifting property with respect to acyclic fibrations. If $L$ and $M$ are two Lie algebras, then the simplicial mapping spaces is given by 
 $$\mathfrak{hom}_n(L,M)=Hom_{s^{-1}Lie}(L,\Omega_n \otimes M).$$
Where $\Omega_n$ is the CDGA of polynomial de Rham forms on the $n$-simplex. 
\end{theorem}

\begin{remark}
 Recall from convention \ref{convallobjectsaredg} that we assumed that all algebras are considered in the category of chain complexes. So in particular when we say Lie algebra we mean differential graded Lie algebra. Also note that Hinich use unshifted Lie algebras, the statement of Theorem \ref{thrmHinichsresult} follows from his results by shifting everything.
\end{remark}


We also need the following lemma.

\begin{lemma}\label{lemsimplicialmodelcategories}
 If $L$ is a shifted Lie algebra of the form $L=\Omega_{s^{-1}Lie}C$, for $C$ some finite-dimensional $C_{\infty}$-algebra and $M$ another shifted Lie algebra. then there is an isomorphism of simplicial sets 
 $$\mathfrak{hom}_{\bullet}(L,M) \cong MC_{\bullet}(Hom_{\K}(C,M)).$$
\end{lemma}

\begin{proof}
 Because of Proposition 11.3.1 in \cite{LV} there is a bijection between the set of shifted Lie algebra homomorphisms between $L$ and $M$ and the set of Maurer-Cartan elements in $Hom_{\K}(C,M)$, this proves that we have a bijection between the zero simplices. To show that we also have bijections between the higher simplices we first note that the set of shifted Lie algebra homomorphisms $Hom_{s^{-1}Lie}(L,\Omega_n \otimes M)$ is isomorphic to the set of Maurer-Cartan elements in $Hom_{\K}(C, \Omega _n \otimes M)$. Since $C$ is finite dimensional, it is straightforward to check that $Hom_{\K}(C, \Omega_n \otimes M)$ is isomorphic to $Hom_{\K}(C, M) \otimes \Omega_n$ and that therefore we have an isomorphism between $MC_n(Hom_{\K}(C,M))$ and $\mathfrak{hom}_n(L,M)$. We leave it to the reader to check that these bijections commute with face and degeneracy maps.
\end{proof}

\begin{proof}[Proof of Theorem \ref{thrmlongexactsequencealgebraiccwcomplexes}]
 To prove Theorem \ref{thrmlongexactsequencealgebraiccwcomplexes} we will first assume that $L$ is a shifted Lie algebra instead of an $s^{-1}L_{\infty}$-algebra. When $L$ is a shifted Lie algebra we have by construction a cofiber sequence of shifted Lie algebras given by  
 $$\Omega_{s^{-1}Lie} \mathfrak{S}^n \rightarrow  \Omega_{s^{-1}Lie} D \rightarrow \Omega_{s^{-1}Lie} C(f).$$
 We can now use Theorem \ref{thrmHinichsresult}, which is Hinich's result that shifted Lie algebras form a simplicial model category. Because of the simplicial model category axiom SM7, we get a fibration sequence of simplicial sets when we apply the functor $\mathfrak{hom}_{\bullet}(-,L)$ to this sequence. This fibration sequence is given by 
 $$\mathfrak{hom}_{\bullet}(\Omega_{s^{-1}Lie} C(f),L) \rightarrow  \mathfrak{hom}_{\bullet}(\Omega_{s^{-1}Lie} D,L) \rightarrow \mathfrak{hom}_{\bullet}(\Omega_{s^{-1}Lie} \mathfrak{S}^n,L).$$
 By Lemma \ref{lemsimplicialmodelcategories} this sequence is isomorphic to the sequence of simplicial sets given by 
 $$MC_{\bullet}(Hom_{\K}(C(f),L)) \rightarrow MC_{\bullet}(Hom_{\K}(D,L)) \rightarrow MC_{\bullet}(Hom_{\K}(\mathfrak{S}^n,L)).$$ 
 So this is a fibration sequence and in particular we get a long exact sequence in homotopy.
 
 To prove Theorem \ref{thrmlongexactsequencealgebraiccwcomplexes} in the case when $L$ is an $s^{-1}L_{\infty}$-algebra, we will first rectify $L$ according to Section 11.4.3 of \cite{LV}. We do this by replacing $L$ by $\Omega_{s^{-1}Lie} B_{s^{-1}L_{\infty}} L$, such that it becomes an actual Lie algebra instead of an $s^{-1}L_{\infty}$-algebra. Then we apply the previous argument to show that 
 $$MC_{\bullet}(Hom_{\K}(C(f),\Omega_{s^{-1}Lie} B_{s^{-1}L_{\infty}} L)) \rightarrow MC_{\bullet}(Hom_{\K}(D,\Omega_{s^{-1}Lie} B_{s^{-1}L_{\infty}}L))$$
 $$ \rightarrow MC_{\bullet}(Hom_{\K}(\mathfrak{S}^n,\Omega_{s^{-1}Lie} B_{s^{-1}L_{\infty}} L)),$$ 
 becomes a fibration sequence of simplicial sets. We now would like to compare this sequence to  
  $$MC_{\bullet}(Hom_{\K}(C(f),L)) \rightarrow MC_{\bullet}(Hom_{\K}(D,L)) \rightarrow MC_{\bullet}(Hom_{\K}(\mathfrak{S}^n,L)).$$
  We will do this by using the $\infty$-quasi-isomorphism $\varphi:\Omega_{s^{-1}Lie} B_{s^{-1}L_{\infty}}L \rightsquigarrow L$. The existence of this $\infty$-morphism is proven in Theorem 11.4.4 of \cite{LV} (this $\infty$-morphism is an $\infty$-quasi-isomorphism since both the twisting morphisms are Koszul). By doing this we get the following diagram of $s^{-1}L_{\infty}$-algebras
  
    $$
 \xymatrix{
 Hom_{\K}(C(f),\Omega_{s^{-1}Lie} B_{s^{-1}L_{\infty}} L) \ar[d] \ar[r]^{ \mbox{    } \varphi_*} & Hom_{\K}(C(f),L) \ar[d] \\
 Hom_{\K}(D,\Omega_{s^{-1}Lie} B_{s^{-1}L_{\infty}}L) \ar[d] \ar[r]^{\mbox{    }  \varphi_*} & Hom_{\K}(D,L) \ar[d] \\
 Hom_{\K}(\mathfrak{S}^n,\Omega_{s^{-1}Lie} B_{s^{-1}L_{\infty}} L) \ar[r]^{\mbox{    }\varphi_*} &   Hom_{\K}(\mathfrak{S}^n,L),
 }
 $$
where the maps $\varphi_*$ are the maps induced by $\varphi$. Since $\varphi$ is a quasi-isomorphism, the induced maps $\varphi_*$ will also be quasi-isomorphisms (see Proposition 5.5 of \cite{RNW1}). Since $C(f)$, $D$ and $\mathfrak{S}^n$ are finite dimensional and $\Omega_{s^{-1}Lie} B_{s^{-1}L_{\infty}}L$ and $L$ are of finite type, we can apply  Theorem \ref{thrmmcthingyorsomething?}, which states that an $\infty_{\alpha}$-quasi-isomorphism of $s^{-1}L_{\infty}$-algebras induces a weak equivalence after applying the functor $MC_{\bullet}$. So we get a commutative diagram of simplicial sets

  $$
 \xymatrix{
 MC_{\bullet}(Hom_{\K}(C(f),\Omega_{s^{-1}Lie} B_{s^{-1}L_{\infty}} L)) \ar[d] \ar[r]^{\varphi_*} & MC_{\bullet}(Hom_{\K}(C(f),L)) \ar[d] \\
 MC_{\bullet}(Hom_{\K}(D,\Omega_{s^{-1}Lie} B_{s^{-1}L_{\infty}}L)) \ar[d] \ar[r]^{\varphi_*} & MC_{\bullet}(Hom_{\K}(D,L)) \ar[d] \\
 MC_{\bullet}(Hom_{\K}(\mathfrak{S}^n,\Omega_{s^{-1}Lie} B_{s^{-1}L_{\infty}} L)) \ar[r]^{\varphi_*} &   MC_{\bullet}(Hom_{\K}(\mathfrak{S}^n,L)).
 }
 $$

 Since the left column is a fibration sequence and all the horizontal maps are weak equivalences, the right column is a homotopy fibration sequence. So in particular it induces a long exact sequence on the level of homotopy groups, which proves the theorem.

 \end{proof}


\part{The Hopf invariants}

\section{Algebraic and rational Hopf invariants}\label{secalgebraichopfinv}

In this section we will recall the definitions of the Sinha-Walter Hopf invariants from \cite{SW2} and the algebraic Hopf invariants from \cite{Wie1}.


In \cite{SW2}, Sinha and Walter define a pairing between the set of homotopy classes of maps and the cohomotopy groups. This is done as follows.

\begin{definition}\label{defHopfpairing}
  Let $X$ be a simply-connected space of finite $\Q$-type and let $X_{\Q}$ be its rationalization. Let $S^n$ be the $n$-dimensional sphere, with $n \geq  2$. There exists a pairing 
 $$\eta:[S^n,X_{\Q}] \otimes \pi^n (X_{\Q}) \rightarrow \Q.$$
 Where $[S^n,X_{\Q}]$ is equipped with the group structure coming from the pinch map of the sphere. The cohomotopy group $\pi^n(X)$ is the $n$th-cohomotopy group from Convention \ref{convcohomotopygroups}, i.e. the $n$th cohomology group of $B_{\iota}A$, for some algebra model $A$ for $X$. This algebra model will be the polynomial de Rham forms in the rational case and the ordinary de Rham forms in the real case. The pairing is defined as follows, given a cohomotopy form $\omega \in \pi^n(X_{\Q})=H^*(A_{PL}^{\bullet}(X))$. We construct a cohomotopy form $f^* \omega$ on the sphere, by defining the form as  the pull back $f^* \omega$ on the sphere $S^n$. The pairing is now given by applying $h^{\vee}$, the dual of the Hurewicz homomorphism (see Definition \ref{defdualhurewicz}) to the cohomotopy form $f^* \omega$. This way we get a cohomology class $h^{\vee}(f^* \omega)$. The pairing is then defined by evaluating the form $h^{\vee}(f^*\omega)$ on the fundamental class of $S^n$, i.e.
 $$\eta(f,\omega)=\int_{S^n}h^{\vee}(f^*\omega).$$
\end{definition}

The following theorem is  Theorem 2.10 of \cite{SW2}.

\begin{theorem}\label{thrmwaltersinhahopfinv}
 The pairing from Definition \ref{defHopfpairing} is a perfect pairing. In particular two maps $f,g \in Map_*(S^n,X_{\Q})$ are homotopic if and only if  we have the following equality:
 $$\int_{S^n}h^{\vee}(f^* \omega) =\int_{S^n} h^{\vee}( g^* \omega )  ,$$
 for all $\omega \in \pi^n(X_{\Q})$.
\end{theorem}

So in particular the Sinha-Walter version of the Hopf invariant can distinguish points in the space $[S^n,X_{\Q}]$, i.e. determine whether two maps are homotopic or not. An algebraic generalization of Theorem \ref{thrmwaltersinhahopfinv} was given in \cite{Wie1} and is given in the following definition, which can be found in Section 9 of \cite{Wie1}.

\begin{definition}\label{coalghopfinv}
 Let $\mathcal{C}$ be a cooperad and let $C$ and $D$  be $\mathcal{C}$-coalgebras. Denote by $\iota:\mathcal{C} \rightarrow \Omega_{op} \mathcal{C}$ the canonical twisting morphism from $\mathcal{C}$ to its cobar construction. Denote by $H_*(C)$ the homology of the coalgebra $C$ together with the transferred structure coming from the Homotopy Transfer Theorem. Let $i:H_*(C) \rightsquigarrow C$ be an $\infty_{\iota}$-quasi-isomorphism, i.e. a $\mathcal{P}$-coalgebra map $I:\Omega_{\iota}H_*(C) \rightarrow \Omega_{\iota}C$. Let $P:\Omega_{\iota} D \rightarrow H_*(\Omega_{\iota}D)$ be a strict morphism of $\Omega_{op}\mathcal{C}$-algebras. Then we define a map 
 $$mc:Hom_{\mathcal{C}\mbox{-coalg}}(C,D) \rightarrow Hom_{\Q}(H_*(C),H_*(\Omega_{\iota}D)),$$
by sending a map $f \in Hom_{\mathcal{C}\mbox{-coalg}}(C,D)$ to the Maurer-Cartan element corresponding to the following composition of maps
$$\Omega_{\iota} H_* (C) \xrightarrow{I} \Omega_{\iota}C \xrightarrow{\Omega_{\iota}f} \Omega_{\iota}D \xrightarrow{P} H_*(\Omega_{\iota} D) .$$
If we use $\theta$ to denote the canonical twisting morphism $\theta:H_*(C) \rightarrow \Omega_{\iota} H_*(C)$, then the map $mc(f)$ is given by 
$$mc(f)=\theta \circ I \circ \Omega_{\iota} f \circ P .$$ 
 The map $mc$  associates a Maurer-Cartan element to each coalgebra morphism $f:C \rightarrow D$. When we compose this map with the projection onto the moduli space of Maurer-Cartan elements, we get the map 
 $$mc_{\infty}:Hom_{\mathcal{C}\mbox{-coalg}}(C,D) \rightarrow \mathcal{MC}(H_*(C),H_*(\Omega_{\iota}D)),$$
 which we will call the algebraic Hopf invariant.
\end{definition}

\begin{remark}
 The existence of a strict morphism $P$ of $\Omega_{op} \mathcal{C}$-algebras is proven in Proposition 9.1  of \cite{Wie1}.
\end{remark}

\begin{theorem}[\cite{Wie1}, Theorem 10.1]\label{thrmcoalhopfinv}
Two maps $f,g:C \rightarrow D$ of $\mathcal{C}$-coalgebras are homotopic if and only if  $mc_{\infty}(f)=mc_{\infty}(g)$.
\end{theorem}

In this paper we will also use the dual version of Definition \ref{coalghopfinv}. The advantage of the dual version of Definition \ref{coalghopfinv} and Theorem \ref{thrmcoalhopfinv} is that we can apply this to the de Rham complex of differential forms. As it turns out that the de Rham complex has the advantage that it has a Hodge decomposition which will be important in Part \ref{partexample} of this paper. Under certain finiteness assumptions, the proofs are completely analogous to the proofs of \cite{Wie1} and will be omitted. Since this definition will be applied to the de Rham complex we will grade everything cohomologically.

\begin{definition}\label{defalghopfinv}
 Let $\mathcal{P}$ be an operad and let $A$ and $B$ be $\mathcal{P}$-algebras. Denote by $\pi:B_{op}\mathcal{P} \rightarrow \mathcal{P}$ the canonical operadic twisting morphism. Let $j:B \rightsquigarrow H^*(B)$ be an $\infty_{\pi}$-quasi-isomorphism from $B$ to its homology $H_*(B)$ equipped with a transferred structure coming from the Homotopy Transfer Theorem. Recall from Convention \ref{convcohomotopygroups}  that $\pi^*(A)$ denotes $H^*(B_{\pi}A)$. Let $q:\pi^*(A) \rightarrow B_{\pi}A$ be a strict morphism of $B_{op} \mathcal{P}$-coalgebras. We define a map
 $$mc^{\vee}:Hom_{\mathcal{P}\mbox{-alg}}(A,B)\rightarrow Hom_{\K}(\pi^*(A),H^*(B)),$$
 as the Maurer-Cartan element corresponding to the following composite
 $$\pi^*(A) \xrightarrow{q} B_{\pi} A \xrightarrow{B_{\pi}f} B_{\pi} B \xrightarrow{B_{\pi}j} B_{\pi} H^*(B).$$
 Again we can use the projection onto the moduli space of Maurer-Cartan elements to define $mc^{\vee}_{\infty}$, which gives us a morphism
 $$mc^{\vee}_{\infty}:Hom_{\mathcal{P}\mbox{-alg}}(A,B) \rightarrow \mathcal{MC}(\pi^*(A),H^*(B)).$$
\end{definition}

\begin{remark}
The existence of the map $q$ can be shown by a similar agument as Proposition 9.1 of \cite{Wie1}.
\end{remark}

\begin{theorem}\label{thrmalghopfinv}
 Two maps $f,g:A \rightarrow B$ of $\mathcal{P}$-algebras are homotopic if and only if $mc^{\vee}_{\infty}(f)=mc^{\vee}_{\infty}(g)$.
\end{theorem}

\begin{remark}
 Note that under certain finiteness assumptions the $s^{-1}L_{\infty}$-algebras $Hom_{\mathcal{P}\mbox{-alg}}(A,B)$ and $Hom_{\K}(H_*(B^{\vee}),H_*(\Omega_{\iota}A)^{\vee})$ are isomorphic and have therefore isomorphic  moduli spaces of Maurer-Cartan elements, see Proposition \ref{proplinftyisomorphism} for more details.
\end{remark}

The definition of the algebraic Hopf invariants has one serious disadvantage, which is the dependency on the choice of maps $I$ and $P$ in the coalgebra case and the dependency on the maps $j$ and $q$ in the algebra case.

In Proposition \ref{propjisformal} and Proposition \ref{propqisformal}, we recall from \cite{Wie1} that the dependency on the maps $q$ and $P$ can be weakened. In particular any quasi-isomorphisms of chain complexes $P:\Omega_{\iota}D \rightarrow H_*(\Omega_{\iota} D)$ and $q:\pi^*(A) \rightarrow B_{\pi} A$ can be turned into strict morphisms. These results are the equivalent of Proposition 9.1 in \cite{Wie1}, since the proofs are completely analogous they will be omitted.

\begin{proposition}\label{propqisformal}
Let $C$ be a $\mathcal{C}$-coalgebra and $\iota:\mathcal{C} \rightarrow \Omega_{op}\mathcal{C}$ be the canonical operadic twisting morphism. Let $P:\Omega_{\iota}C \rightarrow H_*(\Omega_{\iota}C)$ be a quasi-isomorphism of chain complexes. There exists a strict morphism of $\Omega_{op} \mathcal{C}$-algebras $P':\Omega_{\iota} C \rightarrow H_*(\Omega_{\iota} C)$ which has the linear map $P:C \rightarrow H_*(\Omega_{\iota}C)$ as its Maurer-Cartan element. 
\end{proposition}

We will also need the dual version of this proposition.

\begin{proposition}\label{propjisformal}
 Let $A$ be an algebra over an operad $\mathcal{P}$ and let $\pi:B_{op} \mathcal{P}  \rightarrow \mathcal{P}$ be the canonical operadic twisting morphism. Let  $q:H^*(B_{\pi}A) \rightsquigarrow B_{\pi}A$ be a quasi-isomorphism of chain complexes. Then there exists a strict morphism of $B_{op}\mathcal{P}$-coalgebras $q':H^*(B_{\pi}A) \rightarrow B_{\pi}A$, such that $q'$ has $q$ as its Maurer-Cartan element. 
\end{proposition}




\begin{remark}
Note that in the proofs of Proposition \ref{propqisformal} and Proposition \ref{propjisformal} in \cite{Wie1} we also changed the  $\Omega_{op}B_{op}\mathcal{P}$-structure on $H^*(B_{\pi}A)$ (resp. $B_{op} \Omega_{op} \mathcal{C}$-structure on $H_*(\Omega_{\iota}C)$. If we would fix an $\Omega_{op}B_{op}\mathcal{P}$-structure on $H^*(B_{\pi}A)$ and then try to construct the map $q$, Proposition \ref{propjisformal} would not be true. A similar comment holds for Proposition \ref{propqisformal}.
\end{remark}

The explicit formulas for the maps $I$ and $j$ are more complicated and involve the formulas of the Homotopy Transfer Theorem. For our applications we will only need formulas for the map $j:B_{\pi}B \rightarrow B_{\pi}H^*(B)$, so we will only construct the map $j$. 

To construct $j$ we will first pick a contraction diagram as in Definition \ref{defcontraction2}. We do this as follows first we pick a map $j':B \rightarrow H^*(B)$, satisfying the conditions of the map $p$ in Definition \ref{defcontraction2}. Then we complete the diagram by choosing maps $i:H^*(B) \rightarrow B$ and $H:B \rightarrow B$. If we apply the Homotopy Transfer Theorem to this diagram we get an $\infty_{\pi}$-morphism $j:B \rightsquigarrow H^*(B) $, i.e. a map $j:B_{\pi} B \rightarrow B_{\pi}H^*(B)$. We can therefore construct the map $j$ as soon as we have the data of a contraction. In Section \ref{secHodgetheory}, we give an explicit choice for this contraction when the source manifold $M$ is a compact oriented Riemannian manifold without boundary.  From now on we will always assume that such a contraction is fixed.

\section{From spaces to Maurer-Cartan elements}\label{secfromspacestomc}

One of the big problems with Theorem \ref{thrmcoalhopfinv} and Theorem \ref{thrmalghopfinv} is that they do not give us a clear way to associate a Maurer-Cartan element to a smooth based map $f:M\rightarrow N$ of manifolds. The goal of this section is to explain how we can compute the Maurer-Cartan element associated to a map $f$. We explain how to evaluate the morphism $mc:Map_*(X,Y) \rightarrow Hom_{\R}(H_*(X),\pi_*(Y))$ by computing a finite sequence of integrals. By doing this we reduce the problem of deciding whether two maps $f,g:M \rightarrow N$ are real homotopy equivalent to a completely algebraic problem. In practice this algebraic problem is often solvable by straightforward but sometimes very tedious computations. Since these computations can be very tedious we will in Section \ref{secmodulispaces} develop some further methods to avoid these tedious computations in certain special cases. 

There are roughly two approaches to applying the algebraic Hopf invariants to rational homotopy theory. The first approach involves coalgebra models and the second approach involves taking algebra models. In \cite{Wie1}, we used the coalgebra approach by using Quillen's functor $\mathcal{C} \lambda:Top_{*,1} \rightarrow CDGC_{\geq 2}$ (see Theorem 1 of \cite{Quil1}, or Section \ref{secratmodels} of this paper), this is a functor from the category of simply-connected topological spaces to simply-connected  cocommutative coalgebras and induces an equivalence on the level of homotopy categories.

In the previous section we defined the algebraic Hopf invariant maps 
$$mc:Hom_{\mathcal{C}\mbox{-coalg}}(C,D)\rightarrow Hom_{\R}(H_* (C),H_*(\Omega_{\iota} D)),$$
 for maps between $\mathcal{C}$-coalgebras $C$ and $D$ and 
 $$mc^{\vee}:Hom_{\mathcal{P}\mbox{-alg}}(A,B) \rightarrow Hom_{\R}(H_*(B_{\iota}A),H_*(B)),$$
  for maps between $\mathcal{P}$-algebras $A$ and $B$. To extend these maps to a map $mc:Map_*(M,N)\rightarrow Hom_{\R}(H_*(M),\pi_*(N))$ and a map $mc^{\vee}:Map_*(M,N)\rightarrow Hom_{\R}(\pi^*(N),H^*(M))$, we need to pick algebraic models for the manifolds $M$ and $N$. Then we apply the algebraic Hopf invariants to those models. In Section 12 of \cite{Wie1} we did this by using Quillen's functor $\mathcal{C}\lambda:Top_*,1 \rightarrow CDGC$ (see Theorem 1 of \cite{Quil1}). 

The main issue with the functor $\mathcal{C} \lambda$ is that it is not explicit  and therefore not fit to be used for concrete calculations. In this paper we will therefore use the second  approach by using the de Rham complex $\Omega^{\bullet}$, instead of the functor $\mathcal{C} \lambda$. The de Rham complex has the advantage that is very explicit and that is good for doing computations. 

\begin{remark}
Technically the functor $\mathcal{C} \lambda$ is defined as a functor from topological spaces to cocommutative coalgebras over the rationals, we make it in to a functor to cocommutative coalgebras over the real numbers by taking the tensor product with $\R$.
\end{remark}

\begin{definition}
Let $M$ and $N$ be simply-connected manifolds such that $N$ is of finite $\R$-type and $M$ is compact and has a finite dimensional real cohomology ring  $H^*(M;\R)$. Denote by $\Omega^{\bullet}(M)$ and $\Omega^{\bullet}(N)$ the de Rham complexes of $M$ and $N$. The Hopf invariant map $mc^{\vee}:Map_*(M,N) \rightarrow $ $ Hom_{\R}(\pi^*(N), H^*(M))$ is defined by sending a map $f:M \rightarrow N$ to $mc^{\vee}(\Omega^{\bullet}(f))$, the Maurer-Cartan element corresponding to the map $\Omega^{\bullet}(f) : \Omega^{\bullet}(N) \rightarrow \Omega^{\bullet}(M)$. By a small abuse of notation we will denote this map by $mc^{\vee}$ as well.
\end{definition}

\begin{remark}
 Note that we still need to fix the maps $j$ and $q$ from Definition \ref{defalghopfinv}. In Section \ref{secHodgetheory} we will give an explicit description of the  map $j$ and because of Proposition \ref{propqisformal} any choice of inclusion $q':\pi^*(N)\rightarrow B_{\iota} \Omega_{\bullet}(N)$ will define the map $q$.
\end{remark}

\begin{remark}
 It is possible to replace the de Rham forms by other types of rational or real models for the spaces $M$ and $N$. An example of such a model would be the polynomial de Rham forms. As we will see in Section \ref{secHodgetheory}, the  de Rham forms have the important advantage that they have a Hodge decomposition, which will be necessary for obtaining explicit formulas. In Section \ref{secalternativeapproaches} we will briefly discus how some of these other approaches work. 
\end{remark}

The relation between the coalgebra and algebra approach is the following.  In the coalgebra approach we define a map $mc:Map_*(M,N) \rightarrow Hom_{\R}(H_*(M),\pi_*(N))$ and in the algebra approach we define a map $mc^{\vee}:Map_*(M,N) \rightarrow Hom_{\R}(\pi^*(N),H^*(M))$. It turns out that under some finiteness conditions the  $s^{-1}L_{\infty}$-algebras $Hom_{\R}(H_{*}(M),\pi_*(N))$ and  $Hom_{\R}(\pi^*(N),H^*(M))$ are canonically isomorphic. We will prove this in the following proposition. 

\begin{proposition}\label{proplinftyisomorphism}
Let $M$ and $N$ be manifolds, such that $H_*(M)$ is finite dimensional and $\pi_*(N)$ is of finite $\R$-type. Assume that $H_*(M)$ has a $C_{\infty}$-coalgebra structure that makes it into a model for $M$. Further assume that $\pi_*(N)$ has an $s^{-1}L_{\infty}$-algebra structure that makes into a model for $N$. Because of the finiteness assumptions $H^*(M)=H_*(M)^{\vee}$ is a $C_{\infty}$-algebra model for $M$ and $\pi^*(N)=\pi_*(N)^{\vee}$ is an $s^{-1}L_{\infty}$-coalgebra model for $N$. Let $\alpha:C_{\infty}^{\vee} \rightarrow s^{-1}L_{\infty}$ be the Koszul twisting morphism from the $C_{\infty}$-cooperad to the $s^{-1}L_{\infty}$-operad. Denote by $\alpha^{\vee}:s^{-1}L_{\infty}^{\vee} \rightarrow C_{\infty}$ the dual twisting morphism. Then there is a canonical isomorphism of $s^{-1}L_{\infty}$-algebras $\varphi:Hom_{\R}(H_*(M),\pi_*(N))\rightarrow Hom_{\R}(\pi^*(N),H^*(M))$, given by sending a map $f:H_*(M) \rightarrow \pi_*(N)$ to its dual $f^{\vee}:\pi^*(N) \rightarrow H^*(M)$.
\end{proposition}

\begin{remark}
 It is a straightforward check that the dual of a twisting morphism is again a twisting morphism. See also Lemma 7.4 of \cite{RNW1}.
\end{remark}

\begin{proof}
To prove the proposition we first observe that because of our finiteness assumptions the map $\varphi$ is an isomorphism of graded vector spaces. To show that $\varphi$ is an isomorphism of $s^{-1}L_{\infty}$-algebras as well,  we need to show that $\varphi$ commutes with the $s^{-1}L_{\infty}$-operations $l_n$. 

To shorten the notation a bit we will denote $H_*(M)$ by $C$, $H^*(M)$ by $C^{\vee}$, $\pi_*(N)$ by $L$ and $\pi^*(N)$ by $L^{\vee}$. According to the proof of Theorem 7.1 of \cite{Wie1}, the operation $$l_n:Hom_{\R}(C,L)^{\otimes n} \rightarrow Hom_{\R}(C,L)$$
is given by 
$$C \xrightarrow{\Delta_n} C_{\infty}^{\vee}(n) \otimes C^{\otimes n} \xrightarrow{\sum_{\sigma \in S_n}\tau \otimes f_{\sigma (1)} \otimes ... \otimes f_{\sigma (n)}} s^{-1}L_{\infty}(n) \otimes L^{\otimes n} \xrightarrow{\gamma_n} L,$$
where $\Delta_n:C \rightarrow C_{\infty}^{\vee}(n) \otimes C$ is the arity $n$ part of the coproduct of $C$ and $\gamma_n:s^{-1}L_{\infty}\otimes L^{\otimes n} \rightarrow L$ is the arity $n$ part of the $s^{-1}L_{\infty}$ structure on $L$. It is straightforward to check that the dual of this map is equal to 
$$L^{\vee} \xrightarrow{\gamma_n^{\vee}} s^{-1}L^{\vee}_{\infty} \otimes (L^{\vee})^{\otimes n} \xrightarrow{\sum_{\sigma \in S_n}\tau^{\vee} \otimes f^{\vee}_{\sigma (1)} \otimes ... \otimes f^{\vee}_{\sigma (n)}} C_{\infty} \otimes (C^{\vee})^{\otimes n} \xrightarrow{\Delta^{\vee}_n} C^{\vee}.$$
This is the same as the operation $l_n:Hom_{\R}(L^{\vee},C^{\vee})^{\otimes n} \rightarrow Hom_{\R}(L^{\vee},C^{\vee})$ applied to the maps $f_1^{\vee},...,f_n^{\vee}$. The morphism $\varphi$ therefore commutes with the $s^{-1}L_{\infty}$-structures and is therefore an isomorphism of $s^{-1}L_{\infty}$-algebras which proves the proposition.
\end{proof}

 Suppose that we have fixed the maps $j$  and $q$ from Definition \ref{defalghopfinv}, then to every map $f:M \rightarrow N$ we associate the following composition of maps
$$\pi^*(N)=H^*(B_{\iota}\Omega^{\bullet}(N)) \xrightarrow{j} B_{\iota} \Omega^{\bullet}(N) \xrightarrow{B_{\iota} \Omega^{\bullet}(f)} B_{\iota}\Omega^{\bullet}(M) \xrightarrow{i} B_{\iota} H^*(M) .$$ 
So we get a strict morphism $j:\pi^*(N) \rightarrow B_{\iota} H^*(M)$ and a  linear map $mc^{\vee}(f):\pi^*(N) \rightarrow H^*(M)$. If we assume that $H^*(M)$ is finite dimensional and $\pi^*(N)$ is of finite type then we can dualize this map to get a strict morphism $\Omega_{\iota} H_*(M) \rightarrow \pi_*(N)$, which is defined by a  linear map $mc(f):H_*(M) \rightarrow \pi_*(N)$. So even though $\Omega^{\bullet}M$ and $\Omega^{\bullet}N$ are not of finite type we can still dualize the linear map $mc^{\vee}(f)$.  

This way we get an identification between the Maurer-Cartan elements in $Hom_{\R}(H_*(M) ,\pi_*(N))$ and $Hom_{\R}(\pi^*(N),H^*(M))$. So even though we cannot compute the Maurer-Cartan element $mc(f)$ directly, we can compute it indirectly by computing $mc^{\vee}$.


An immediate corollary of Proposition \ref{proplinftyisomorphism}, is the following corollary whose proof we omit.

\begin{corollary}
 If $H^*(M)$ is finite dimensional and $\pi^*(N)$ is of finite $\R$-type, then we have an isomorphism of $s^{-1}L_{\infty}$-algebras
 $$Hom_{\R}(H_*(M) ,\pi_*(N)) \cong Hom_{\R}(\pi^*(N),H^*(M)) .$$
 In particular we get a bijection 
 $$MC_0(Hom_{\R}(H_*(M) ,\pi_*(N))) \cong MC_0(Hom_{\R}(\pi^*(N),H^*(M)))$$
 between the sets of Maurer-Cartan elements.
\end{corollary}

\begin{convention}
 Since the $s^{-1}L_{\infty}$-algebras $Hom_{\R}(H_*(M),\pi_*(N))$ and  $Hom_{\R}(\pi^*(N),H^*(M))$ are canonically isomorphic, we will from now on identify them and denote both the algebraic Hopf invariant maps by $mc$, instead of $mc$ and $mc^{\vee}$. We will further assume that the target of both maps is  $Hom_{\R}(H_*(M),\pi_*(N))$, i.e. the map $mc:Map_*(M,N))\rightarrow Hom_{\R}(H_*(M).\pi_*(N))$  will denote the composition $\varphi \circ mc^{\vee}$, where $\varphi$ is the isomorphism from Proposition \ref{proplinftyisomorphism}.
\end{convention}

\begin{remark}
 Note that we only need to fix the maps $j$ and $q$, if we would also fix $I$ and $P$ we would not necessarily get an isomorphism of $s^{-1}L_{\infty}$-algebras.
\end{remark}

Now that we have established that $mc^{\vee}$ and $mc$ are essentially the same map, we still need to compute them. To do this we use the canonical pairing
$$\eta:Hom_{\R}(H_*(M),\pi_*(N))\otimes \left(H_*(M) \otimes \pi^*(N)\right)\rightarrow \R$$
given by 
$$\eta(f,\alpha,\otimes\omega)= f^*\omega(\alpha) ,$$
For $f\in Hom_{\R}(H_*(M),\pi_*(N))$, $\alpha \in H_*(M)$ and $\omega \in \pi^*(N)$. Since we assumed that $H_*(M)$ is finite dimensional and $\pi_*(N)$ is of finite type this pairing is perfect. So if we fix a map $f:M \rightarrow N$ we  get a pairing between $H_*(M)$ and $\pi^*(N)$.

\begin{definition}
 Let $f:M \rightarrow N$ be a map between simply-connected smooth manifolds, such that $M$ is compact and $H_*(M)$ is finite dimensional and $\pi_*(N)$ is of finite $\R$-type. Then we define the pairing
 $$\eta_f:H_*(M) \otimes \pi^*(N) \rightarrow \R$$
 by 
 $$\eta_f(\alpha,\omega)=(mc(f)^*(\omega))(\alpha) .$$
\end{definition}

Since the pairing $\eta$ is perfect we get a basis for $Hom_{\R}(H_*(M),\pi_*(N))$ by picking  bases for $H_*(M)$ and $\pi^*(N)$. To do this, let $\{\alpha_i\}_{i \in I}$ be a basis for $H_*(M)$ and $\{\omega_j\}_{j \in J}$ a basis for $\pi^*(N)$. Denote by $\{\omega_j^{\vee}\}_{j\in J}$ the dual basis for $\pi^*(N)^{\vee}=\pi_*(N)$. A basis for $Hom_{\R}(H_*(M),\pi_*(N))$ is then given by $\{\varphi_{i,j}\}_{i\in I,j \in J}$, where $\varphi_{i,j}$ is the function $\varphi_{i,j}:H_*(M) \rightarrow \pi_*(N)$ which is defined by $\varphi_{i,j}(\alpha_i)=\omega_j^{\vee}$ and zero otherwise.

The Maurer-Cartan element $mc(f)$, of a map $f:M \rightarrow N$  between manifolds can now be expressed in the basis $\{\varphi_{i,j}\}$. So in particular $mc(f)=\sum_{i,j} \lambda_{i,j}\varphi_{i,j}$ for  coefficients $\lambda_{i,j} \in \R$. 

In the following lemma we will show how we can compute these coefficients $\lambda_{i,j}$ in terms of certain integrals. But first we need to introduce some notation.

To shorten notation a bit we will from now on denote $\Omega^{\bullet}(M)$ by $B$, $\Omega^{\bullet}(N)$ by $A$ and by abuse of notation we will denote the induced map $\Omega^{\bullet}(f):A \rightarrow B$ by $f$. We will also assume that we have a contraction of vector spaces between $B$ and $H^*(B)$, i.e. we have a diagram  

$$\xymatrix{
B \ar@(ul,dl)[]|{h} \ar@/^/[rr]|p
&& H^*(B), \ar@/^/[ll]|{i} }$$

such that the maps satisfy the conditions from Definition \ref{defcontraction2} and Theorem \ref{thrmHTT}. Finally we will also assume we  that we have a linear map $q:H^*(B_{\pi}A) \rightarrow  B_{\pi}A$. We assume that all these maps are just linear maps and not compatible with any of the algebraic structures. 

In practice these assumptions are not very restrictive. In Section \ref{secHodgetheory}, we will explain that when the space $M$ is a compact oriented Riemannian manifold, there are canonical choices for the maps $i$, $p$ and $h$. The map $q:H^*(B_{\pi}A) \rightarrow B_{\pi}A$ is just an explicit choice of a set of cocycles in $B_{\pi}A$ representing the cohomotopy of $A$, i.e. every element $\omega \in H^*(B_{\pi}A)$ will be send to a linear combination of elements  the form $\nu \otimes a_1 \otimes ...\otimes a_n$, with $\nu \in \mathcal{C}(n)$ and $a_i \in A$. Using these assumptions we get the following formulas for the coefficients $\lambda_{i,j}$.

\begin{lemma}\label{lemconcreteformulas}
 Let $f:M \rightarrow N$ be a map and let $\omega \in \pi^*(N)$ and $\alpha \in H_*(M)$,  the pairing $\eta_f$ between $\omega$ and $\alpha$ is given by the following formula:
 $$\eta_f(\alpha,\omega)=\sum_{n \geq 1} \int_{\alpha} \left((-1)^{\mid \nu \mid} p t^{\nu} \mathbf{h}_n(f^*q_n(\omega)) + \sum_{i=1}^{u} (-1)^{\mid \nu_i '' \mid}(p^{\nu_i '} \circ_{e_i} t^{ \nu_i ''}) \tau_i \mathbf{h}_n (f^*q_n(\omega)) \right)  .$$
 where $q_n(\omega)$ is the weight $n$ component of $q(\omega)$, the second sum runs over the coproduct of the $s^{-1}L_{\infty}$-operad and the maps $\mathbf{h}_n$ are as in Section \ref{secHTT} (see Theorem \ref{thrmHTT} and Section \ref{secHTT} for more details).
\end{lemma}
  
\begin{remark}
 The sum in Lemma \ref{lemconcreteformulas} converges since $q(\omega)$ is non-zero in only finitely many different weights. 
\end{remark}

\begin{proof}
 The proof of the lemma is a straightforward application of the Homotopy Transfer Theorem. The pairing $\eta_f(\alpha,\omega)$ is defined as the composite $\eta_f(\alpha,\omega)=\int_{\alpha} j f^* q (\omega)$, the formulas from Lemma \ref{lemconcreteformulas} are obtained by writing down the explicit formulas for the map $P$ coming from the Homotopy Transfer Theorem.
\end{proof}

\begin{theorem}\label{thrmhowtocomputeMaurarCartanelements}
 Let $\{\alpha_i\}_{i \in I}$ be a basis for $H_*(M)$, $\{\omega^{\vee}_j\}_{j \in J}$ a basis for $\pi_*(N)$ and let $\{\varphi_{ij}\}_{i \in I, j \in J}$ be a basis for $Hom_{\R}(H_*(M),\pi_*(N))$. Further assume that the basis $\{\omega^{\vee}_j\}$ is dual to basis $\{\omega_j\}$ for $\pi^*(N)$, i.e. $\omega_j(\omega^{\vee}_k)=\delta_{jk}$, where $\delta_{jk}$ is the Kronecker delta. Let $mc(f)=\sum_{i,j} \lambda^f_{ij} \varphi_{ij}$, the coefficient $\lambda^f_{ij}$ can the be computed as the following integral
 $$\lambda^f_{ij}=\eta_f(\alpha_i, \omega_j)=\sum_{n \geq 1} \int_{\alpha_i} \left((-1)^{\mid \nu \mid} p t^{\nu} \mathbf{h}_n(f^*q_n(\omega_j)) + \sum_{k=1}^{u} (-1)^{\mid \nu_k '' \mid}(p^{\nu_k '} \circ_{e_k} t^{ \nu_k ''}) \tau_k \mathbf{h}_n (f^*q_n(\omega_j)) \right)  .$$
\end{theorem}

\begin{proof}
 To prove the theorem we observe that we have a perfect pairing 
 \[\left<,\right>:Hom_{\R}(H_*(M),\pi_*(N))\otimes \big( H_*(M)\otimes \pi^{*}(N) \big) \rightarrow  \R\]
  given by $\left<\varphi_{ij},\alpha_k \otimes \omega_l\right> =\delta_{ik} \delta_{jl}$. In particular the  basis $\{\alpha_i \otimes \omega_j\}$ is dual to the basis $\{\varphi_{ij}\}$, so the coefficient $\lambda^f_{ij}$ of the basis element $\varphi_{ij}$ can be computed by evaluating $\alpha_i \otimes \omega_j$ on $\varphi_{ij}$. Therefore $\lambda^f_{ij}$ is equal to $\eta_f(\alpha_i,\omega_j)$, which proves the theorem.
\end{proof}

The significance of Theorem \ref{thrmhowtocomputeMaurarCartanelements} is that it gives us a way to reduce the problem of deciding whether two maps $f,g:M \rightarrow N$ are homotopic to deciding whether the Maurer-Cartan elements $mc(f)$ and $mc(g)$ are gauge equivalent in the $L_{\infty}$-algebra $Hom_{\R}(H_*(M),\pi_*(N))$. Since $Hom_{\R}(H_*(M),\pi_*(N))_0$ is finite dimensional it is often possible to decide this in practice, unfortunately these calculations can become extremely tedious. In the next section we will give some methods that can help us in special cases to get a better understanding of the moduli space of Maurer-Cartan elements.

\section{Algebraic CW-complexes and the moduli space of \\ Maurer-Cartan elements}\label{secmodulispaces}

Let $f:M \rightarrow N$ be a map, in the previous section we explained how to compute the coefficients $\lambda^f_{ij}$ of $mc(f)$ of the map $f$ by computing  certain integrals. Unfortunately the coefficients $\lambda^f_{ij}$ are not an invariant of the homotopy class of $f$. It is therefore necessary to compute the moduli space of Maurer-Cartan elements $\mathcal{MC}(Hom_{\R}(H_*(M),\pi_*(N)))$. This can be done in general but might be very tedious and will involve solving many equations. In this section we will present an alternative approach based on the algebraic CW-complexes and the long exact sequence from Theorem \ref{thrmlongexactsequencealgebraiccwcomplexes}. This approach gives us some information about the moduli space of Maurer-Cartan elements. In many cases this method will be good enough to completely determine whether two maps are homotopic or not. The approach described in this section can be seen as an extension of some of the ideas of Stasheff and Schlessinger from \cite{SS2} to mapping spaces.

The idea is that if we  have an algebraic CW-complex $C$, we can filter $C$ by its skeleta.  Denote by $C_{\leq n}$ the $n$-skeleton of $C$. Let $L$ be an $s^{-1}L_{\infty}$-algebra. We can use the skeletal filtration of $C$ to obtain a tower of fibrations 
$$MC_{\bullet}(Hom_{\R}(C_{\leq n},L)) \rightarrow MC_{\bullet}(Hom_{\R}(C_{\leq n-1},L)) \rightarrow ... \rightarrow MC_{\bullet}(Hom_{\R}(C_2,L)).$$
Suppose that we have two maps $f,g:M \rightarrow N$, then we can first compare the $MC_{\bullet}(C_2,L)$ parts of the maps $f$ and $g$. If these are not the same, the maps are certainly not homotopic. If the $\mathcal{MC}(Hom_{\R}(C_2,L))$ parts are the same then we check the $MC_{\bullet}(Hom_{\R}(C_{\leq 3},L))$ parts and so on.

To make this more precise we will first introduce some definitions and notation. For simplicity we will only work with minimal CW-complexes in this section, we will now recall the definition of a minimal CW-complex.

\begin{definition}
 A CW-complex $X$ is called minimal if the differential of the cellular chain complex with integer coefficients is equal to zero.  Similarly an algebraic CW-complex $C$ is called minimal if $d_C=0$.
\end{definition}

\begin{theorem}\label{thrmminimlcwcomplex}
 Let $M$ be a simply-connected manifold, then $M$ is rationally equivalent to a minimal CW-complex and therefore can be modeled by a minimal algebraic CW-complex. 
\end{theorem}

\begin{proof}
 The first part follows from Theorem 9.11 in \cite{FHT}. The second part follows from Corollary 2.14 of \cite{BL1}, which states that every simply-connected space $M$ has a minimal shifted Lie model of the form $(s^{-1}\mathcal{LIE}(H_*(M)),d)$, for a certain differential $d$. Since the existence of the minimal model  $(s^{-1}\mathcal{LIE}(H_*(M)),d)$ is the same as an algebraic CW-decomposition for  $H_*(M)$. Since $d_{H_*(M)}=0$, this implies that $H_*(M)$ has a minimal algebraic CW-decomposition.
\end{proof}

In this section we will from now on assume that $C$ is a $1$-reduced minimal algebraic CW-complex modeling $X$ and that $L$ is a minimal $s^{-1}L_{\infty}$-algebra modeling $Y$, i.e. $d_C=d_L=0$. These are not very severe restrictions, because of Theorem \ref{thrmminimlcwcomplex} every topological space is rationally equivalent to a minimal CW-complex.  The $s^{-1}L_{\infty}$-algebra $L$ can always be obtained by taking the homotopy groups of $Y$, with the appropriate $s^{-1}L_{\infty}$-structure. It is possible to drop the restriction that $C$ and $L$ are minimal, but this will make statements and techniques in this more tedious and less effective.

If $C$ is an algebraic CW-complex, we denote the number of $n$-cells by $k_n$. The  attaching map of the $(n+1)$-cells will be denoted  by $a_n:\bigoplus_{k_{n+1}} \mathfrak{S}^{n} \rightsquigarrow C_{\leq n}$. We will denote the inclusion map of the $n$-skeleton into the $(n+1)$-skeleton by $i_n:C_{\leq n} \rightarrow C_{\leq n+1}$.

\begin{proposition}
 Each attaching map $a_n$ induces a homotopy cofiber sequence of $C_{\infty}$-coalgebras
 $$ \bigoplus_{k_{n+1}}\mathfrak{S}^n \xrightarrow{a_n} C_{\leq n} \xrightarrow{i_n} C_{\leq n+1}.  $$
\end{proposition}

\begin{proof}
 Since $C_{\leq  n+1}$ is the mapping cone of $a_n$ it is the homotopy cofiber of the map $a_n$. The sequence is therefore a homotopy cofiber sequence.
\end{proof}

After applying the functor $Hom_{\R}(-,L)$ we obtain a fibration sequence.

\begin{proposition}
 Let $L$ be a simply-connected $s^{-1}L_{\infty}$-algebra, then we have the following fibration sequence of $s^{-1}L_{\infty}$-algebras 
 $$Hom_{\R}(C_{\leq n+1},L) \twoheadrightarrow Hom_{\R}(C_{\leq n},L) \rightarrow Hom_{\R}(\bigoplus_{k_{n+1}}\mathfrak{S}^n,L).$$ 
\end{proposition}

From now on we will use $p_n$ to denote the fibration  $p_n:Hom_{\R}(C_{\leq n+1},L) \rightarrow Hom_{\R}(C_{\leq n},L) $.

\begin{proof}
Since the sequence $ \bigoplus_{k_{n+1}}\mathfrak{S}^n \xrightarrow{a_n} C_{\leq n} \xrightarrow{i_n} C_{\leq n+1}$, is a cofibration sequence and the functor $Hom_{\R}(-,L)$ turns cofibration sequences into fibration sequences we get a fibration sequence.
\end{proof}

So using this proposition we get a tower of fibrations of Maurer-Cartan simplicial sets. Because of Lemma 1.4.6 of \cite{MP11} we can shift this fibration, to a fibration sequence
$$Hom_{\R}(\bigoplus_{k_{n+1}}\Sigma \mathfrak{S}^n,L) \rightarrow Hom_{\R}(C_{\leq n+1},L) \twoheadrightarrow Hom_{\R}(C_{\leq n},L) .$$
When we look at the corresponding tower, we get the following.

$$
 \xymatrix{
 {}   & ... \ar[d]^{p_{n+1}} \\
 MC_{\bullet}(Hom_{\R}(\bigoplus_{k_{n+2}} \mathfrak{S}^{n+1}),L) \ar[r] & MC_{\bullet}(Hom_{\R}(C_{\leq n+1},L)) \ar[d]^{p_{n}} \\
MC_{\bullet}( Hom_\R(\bigoplus_{k_{n+1}} \mathfrak{S}^{n},L) ) \ar[r] & MC_{\bullet}(Hom_{\R}(C_{ \leq n},L)) \ar[d]^{p_{n-1}} \\
{} & ... \ar[d]^{p_{3}} \\ 
 MC_{\bullet}(Hom_{\R}(\bigoplus_{k_{4}} \mathfrak{S}^3,L) )\ar[r] &  MC_{\bullet}( Hom_{\R}(C_{\leq 3},L))\ar[d]^{p_{2}} \\
{} & MC_{\bullet}(Hom_{\R}(C_{2},L)).
 }
 $$

If we assume that the algebraic CW-complex $C$ is of dimension $d$, i.e. $C=C_{\leq d}$, then the tower of fibrations stops when $n \geq d+1$. The Maurer-Cartan elements $mc(f)$ and $mc(g)$ are therefore elements of the space $Hom_{\R}(C_{ \leq d},L)$. 

\begin{definition}
The map  $q_n:MC_{\bullet}(Hom_{\R}(C_{ \leq d},L))\rightarrow MC_{\bullet}(Hom_{\R}(C_{ \leq n},L))$ is defined as the composite $q_n=p_n \circ p_{n+1} \circ ... \circ p_{d-1}$.
\end{definition}

The following lemma is a straightforward consequence of the fact that the maps $p_n$ are fibrations.

\begin{lemma}\label{lemmafibrations}
If $\tau, \kappa \in MC_0(Hom_{\R}(C_{ \leq d},L))$ are gauge equivalent Mau\-rer-Cartan elements, then $q_n(\tau)$ is gauge equivalent to $q_n(\kappa)$ for all $n \leq d$.
\end{lemma}

\begin{proof}
Since we assumed that $\tau$ and $\kappa$ are gauge equivalent, this means that $\tau$ and $\kappa$ are in the same path component of $MC_{\bullet}(Hom_{\R}(C_{\leq d},L))$.  Since the map $q_n$ is a map of simplicial sets, it preserves path components and therefore the elements $q_n(\tau)$ and $q_n(\kappa)$ are in the same path component. Two Maurer-Cartan elements are by definition gauge equivalent if and only if they are in the same path component, so therefore $q_n(\tau)$ and $q_n(\kappa)$ are gauge equivalent.
\end{proof}

As a corollary of Lemma \ref{lemmafibrations} and the tower of fibrations, we obtain  some sort of  filtration on the algebraic Hopf invariant $mc_{\infty}(f)$. 

\begin{definition}\label{defapproximaioonss}
Let $f:M \rightarrow N$ be a map between simply-connected manifolds $M$ and $N$, such that $M$ is compact. Then we define the degree $\leq n$ part of the Maurer-Cartan element $mc(f)$, as the image of $mc(f)$ under the map $q_n$. We will denote the degree $\leq n$ part of $mc(f)$ by $mc^n(f):=q_n \circ mc(f)$. Similarly we also define the the degree $\leq n$ part of the algebraic Hopf invariant $mc_{\infty}(f)$ as the element in the moduli space of Maurer-Cartan elements $\mathcal{MC}(Hom_{\R}(C_{\leq n},L))$ corresponding to $q_n(mc(f))$. We will denote the degree $\leq n$ part of $mc_{\infty}(f)$ by $mc_{\infty}^n(f)$.
\end{definition}

The following lemma is a straightforward consequence of Definition \ref{defapproximaioonss} and its proof will be omitted.

\begin{lemma}
 The element $mc^n_{\infty}(f)$ is an invariant of the homotopy class of the map $f$.
\end{lemma}

Now suppose that we have two map $f,g:M \rightarrow N$. The idea is that, at least for small $n$, the invariants $mc^{n}_{\infty}(f)$ and $mc_{\infty}^{n}(g)$ are easier to compute than $mc_{\infty}(f)$ and $mc_{\infty}(g)$. Since $mc_{\infty}^d(f)=mc_{\infty}(f)$ and $mc_{\infty}^d(g)=mc_{\infty}(g)$, the invariants $mc^n_{\infty}(f)$ and $mc_{\infty}^n(g)$ can therefore be seen as approximations to $mc_{\infty}$ which get better as $n$ increases. 

We would now like to compute these approximations using the tower of fibrations. The first approximations $mc_{\infty}^2(f)$ and $mc_{\infty}^2(g)$ are easy to compute and this will be done in the following lemma.

\begin{lemma}
 The quotient map $\pi:MC_0(C_2,L) \rightarrow \mathcal{MC}(C_2,L)$ is an isomorphism.  The first approximation $mc_{\infty}^2(f)$ is equal to $mc^2(f)$ under this isomorphism and similarly $mc_{\infty}^2(g)$ is equal to $mc^2(g)$. In other words this mean that the coefficients $\lambda_{i,j}^f$ and $\lambda_{i,j}^g$ of $mc^2(f)$ and $mc^2(g)$ are invariants of the maps $f$ and $g$, where $i$ runs over a basis for $C_2$ and $j$ runs over a basis for $L_2$.
\end{lemma}

\begin{proof}
 To prove the lemma, we need to show that the moduli space of Maurer-Cartan elements $\mathcal{MC}(Hom_{\R}(C_2,L))$ is isomorphic to the set of Maurer-Cartan elements in $MC_0(Hom_{\R}(C_2,L))$. Since the $C_{\infty}$-coalgebra $C_2$ is concentrated in degree $2$, it has the trivial $C_{\infty}$-structure. The convolution $s^{-1}L_{\infty}$-algebra $Hom_{\R}(C_2,L)$ is therefore abelian and has zero differential, which implies that every element in $Hom_{\R}(C_2,L)_0$ is a Maurer-Cartan element and that there is no non trivial gauge equivalence. Therefore we have an isomorphism between $\mathcal{MC}(Hom_{\R}(C_2,L))$ and  $MC_0(Hom_{\R}(C_2,L))$, so two Maurer-Cartan elements are gauge equivalent if and only if they are equal. Therefore we have an equality between $mc_{\infty}^2(f)$ and $mc^2(f)$ and an equality between $mc_{\infty}^2(g)$ and $mc^2(g)$.
\end{proof}

Now that we know how to compute $mc^2_{\infty}(f)$ and $mc_{\infty}^2(g)$, we proceed by induction. To do this we will use the long exact sequence from Theorem \ref{thrmlongexactsequencealgebraiccwcomplexes} and apply it to the fibration $p_2:MC_{\bullet}(Hom_{\R}(C_{\leq 3},L)) \rightarrow MC_{\bullet}(Hom_{\R}(C_2,L))$. The idea is that whenever $mc^2(f)=mc^2(g)$, then since $\pi_1(Hom_{\R}(\bigoplus_{k_{n+1}}\mathfrak{S}^n,L))$ acts transitively on the fiber of above $MC^2(f)$. The elements $mc^3(f)$ and $mc^3(g)$ will differ by an element $\gamma\in \pi_1(MC_{\bullet}(Hom_{\R}(\bigoplus_{k_3}\mathfrak{S}^2,L)))$ of the fundamental group of the fiber. This element $\gamma $ is not unique, but different choices of $\gamma$ will differ by elements of a certain subspace, which we have very good control over. The equivalence class of $\gamma$ will be an invariant of the maps $f$ and $g$. We will make this precise as follows.

\begin{proposition}\label{propsomethingsomething}
Let $f,g:M \rightarrow N$ be maps such that $mc^2(f)=mc^2(g)=\tau$, for a certain Maurer-Cartan element  $\tau \in MC_0(Hom_{\R}(H_2(M),\pi_2(N))$. The group $\pi_1(Hom_{\R}(\bigoplus_{k_3} \mathfrak{S}^2,L))$ acts on the set of Maurer-Cartan elements $MC_{\bullet}(Hom_{\R}(C_{\leq 3},L))$, such that there exists a $\gamma \in \pi_1(Hom_{\R}(\bigoplus_{k_3} \mathfrak{S}^2,L))$ such that $\gamma \cdot mc^3(f)=mc^3(g)$. The two Maurer-Cartan elements $mc^3(f)$ and $mc^3(g)$ are gauge equivalent if and only if $\gamma$ is an element of the subspace $Im((a_3^*)^{\tau})$, where $(a_3^*)^{\tau}$ is the map induced by the attaching map $a_3:\bigoplus \mathfrak{S}^2 \rightarrow C_{\leq 2}$ twisted by the Maurer-Cartan element $\tau$.
\end{proposition}

\begin{proof}
 We will first prove this proposition when $mc^2(f)=mc^2(g)=0$. Since $0$ is the base point of $MC_{\bullet}(Hom_{\R}(C_{\leq 3},L)$ and $\tau$ is equal to $0$, we can  use the long exact sequences from Theorem \ref{thrmlongexactsequencealgebraiccwcomplexes} and  Lemma \ref{lemlesfibration}. Lemma \ref{lemlesfibration} now states that $\pi_1(Hom_{\R}(\bigoplus_{k_3}\mathfrak{S}^3,L))$ acts transitively on the fiber. Since $mc^2(f)=mc^2(g)$, the elements $mc^3(f)$ and $mc^3(g)$ lie in the same fiber. So in particular there is an element  $\gamma \in   \pi_1(Hom_{\R}(\bigoplus_{k_3}\mathfrak{S}^3,L))$, such that $\gamma \cdot mc^3(f)=mc^3(g)$. Lemma \ref{lemlesfibration} also states that, if $mc^3(f)$ and $mc^3(g)$ are elements of the fiber over the base point, then if $\beta, \gamma \in \pi_1(Hom_{\R}(\bigoplus_{k_3}\mathfrak{S}^3,L))$ such that $\gamma \cdot mc^3(f)=mc^3(g)$ and $\beta \cdot mc^3(f)=mc^3(g)$, then $\beta$ and $\gamma$ are gauge equivalent if and only if they differ by an element $\alpha \in \pi_1(Hom_{\R}(C_{\leq 2},L))$. So in particular, the fiber over the base point is isomorphic to $\pi_1(Hom_{\R}(\bigoplus_{k_3}\mathfrak{S}^3,L))/Im(a_n^*)$. The elements $mc^3(f)$ and $mc^3(g)$ are therefore gauge equivalent if and only if $\gamma\in Im(a_n^*)$. This proves the proposition when $\tau$ is the base point.
 
 When $mc^2(f)=mc^2(g) =\tau \neq 0$, we will use the twist from Theorem \ref{thrmchangeofbasepoint} to change the base point. In this case we get a fibration sequence of the twisted $s^{-1}L_{\infty}$-algebras
 $$Hom_{\R}(C_{\leq 3},L)^{\tau} \rightarrow Hom_{\R}(C_{\leq 2},L)^{\tau} \rightarrow Hom_{\R}(\bigoplus_{k_3}\mathfrak{S}^2,L)^{\tau}.$$
 Since we twisted all the $s^{-1}L_{\infty}$-algebras, the elements $mc^2(f)$ and $mc^2(g)$ now both map to the base point. We can therefore use the arguments from the previous part of this proof. It therefore follows that that the Maurer-Cartan elements $mc^3(f)$ and $mc^3(g)$ are gauge equivalent if and only if they differ by an element $\gamma \in Im((a_3^*)^{\tau})$.
\end{proof}

Because of this proposition, we can speak about the difference of two Maurer-Cartan elements $mc^3(f)$ and $mc^3(g)$ as long as they belong to the same fiber, i.e. $mc^2(f)=mc^2(g)$.


It is easy to generalize these arguments to the higher $n$, we will therefore proceed by inductively defining the higher invariants. With this we mean, suppose that we have computed $mc^n_{\infty}(f)$ and $mc^n_{\infty}(g)$ and they are the same, then we would like to compute $mc_{\infty}^{n+1}(f)$ and $mc_{\infty}^{n+1}(g)$ as we did with $mc^2$ and $mc^3$. There is one problem with this approach and that is that it can now be the case that $mc^n(f)$ and $mc^n(g)$ are gauge equivalent but are not the same. In this case our method fails and additional calculations are necessary. Fortunately our method works in many special cases as we will demonstrate in Section \ref{subsecsomexamples}. In the first example we will give, we give a complete invariant of maps from $S^n \times S^m$ to a general space $Y$. In the second example we will show how we can detect the real homotopy class of the constant map using this method. We first explain how to compute $mc^{n+1}_{\infty}(f)$ and $mc_{\infty}^{n+1}(g)$. 

\begin{proposition}\label{proppropprop}
 Let $f,g:M \rightarrow N$ be maps and  assume that we have computed $mc^n_{\infty}(f)$ and $mc_{\infty}^n(g)$. Further assume that $mc^n(f)=mc^n(g)=\tau$ for some $\tau \in MC_0(Hom_{\R}(C_{\leq n},L))$. The group  $\pi_1(Hom_{\R}(\bigoplus_{k_{n+1}}\mathfrak{S}^{n},L))$ acts transitively on the fiber over the point $\tau$, i.e. there exists a $\gamma \in \pi_1(Hom_{\R}(\bigoplus_{k_{n+1}} \mathfrak{S}^n,L))$ such that $\gamma \cdot mc^{n+1}(f)=mc^{n+1}(g)$.  The Maurer-Cartan elements  $mc^{n+1}(f)$ and $mc^{n+1}(g)$ are gauge equivalent if and only if $\gamma \in Im((a_{n+1}^*)^{\tau})$.
\end{proposition}

\begin{remark}
Note that since $\mathfrak{S}^n$ is abelian and $L$ is minimal, the group $\pi_1(Hom_{\R}(\bigoplus_{k_{n+1}}\mathfrak{S}^{n},L))$  is isomorphic to  $\bigoplus_{k_n} L_n\cong \bigoplus_{k_{n+1}}\pi_n(N) \cong \bigoplus_{k_{n+1}} H_n(L)$.
\end{remark}

The proof is completely analogous to the proof of Proposition \ref{propsomethingsomething} and will be omitted.

\begin{remark}
 For the method to work we need to assume that $mc^n(f)=mc^n(g)$ instead of $mc^n(f)$ and $mc^n(g)$ to be just gauge equivalent. If $mc^n(f)$ and $mc^n(g)$ are gauge equivalent but not equal this method will not work and it is necessary to do more explicit computations on the moduli space of Maurer-Cartan elements. In theory it should be possible to compare the fibers above gauge equivalent Maurer-Cartan elements but this might not necessarily be simpler than just computing the moduli space of Maurer-Cartan elements.
\end{remark}

Similar to the degree $2$ and $3$ case we can speak about the difference between two Maurer-Cartan elements $mc^{n+1}(f)$ and $mc^{n+1}(g)$ if they lie above the the same Maurer-Cartan element $\tau=mc^{n}(f)=mc^n(g)$. The following corollary is a straightforward consequence of Proposition \ref{proppropprop}.

\begin{corollary}
Under the hypotheses of Proposition \ref{proppropprop} the Maurer-Cartan elements $mc^{n+1}(f)$ and $mc^{n+1}(g)$ are gauge equivalent if and only if they differ by an element of $Im((a_{n+1}^*)^{\tau})$.
\end{corollary}

\begin{remark}
Again it is important that $mc^n(f)=mc^n(g)$ and not just that $mc_{\infty}^n(f)=mc_{\infty}^n(g)$.
\end{remark}

What we have essentially done in this section is the following. First we used that there are projection maps from $p_{n}:Hom_{\R}(C_{\leq n+1},L) \rightarrow Hom_{\R}(C_{\leq n},L)$. So if we have a point $q_n(\tau) \in Hom_{\R}(C_{\leq n},L)$ we can describe the inverse image  of $q_n(\tau)$ under the map $p_n$ as a subspace of $Hom_{\R}(C_{\leq n+1},L)$ isomorphic to $Hom_{\R}(C_{n+1},L)$. Then we used the long exact sequence in homotopy to identify the subspace $p^{-1}_n(q_n(\tau))$ with $\pi_1(Hom_{\R}(\bigoplus_{k_{n+1}},L)$.  Then we used the attaching maps of $C$ to determine which elements in $\pi_1(Hom_{\R}(\bigoplus_{k_{n+1}},L)$ are gauge equivalent to each other.

\subsection{Some examples}\label{subsecsomexamples}

We will now give two examples to show how our method can be applied to decide whether maps are homotopic or not. In the first example we will take a relatively simple source space and show that it is fairly easy to decide whether two maps are homotopic or not.

In the second example we how to decide whether a Maurer-Cartan element is gauge equivalent to the zero Maurer-Cartan element or not. In particular we show that in this case it is not necessary to worry about gauge equivalence.

\subsubsection{Maps from $S^n \times S^m$ to a space $Y$}\label{secsxs}

We will now describe how we can obtain a complete invariant of maps from $S^n\times S^m$ to a general space $Y$. To do this we will first need an algebraic CW-decomposition for the homology of  $S^n \times S^m$. 

The following lemma is well known and its proof will be omitted. Recall that we assumed that all the homology is always reduced.

\begin{lemma}
The homology of $S^n \times S^m$ is given by the coalgebra $H_*(S^n \times S^m)$, which has as basis one element $\alpha$ of degree $n$, one element $\beta$ of degree $m$ and one element $\gamma$ of degree $n+m$. The coproduct is given by $\Delta(\gamma)=\alpha \otimes \beta +(-1)^{\mid \alpha \mid \mid \beta \mid}\beta \otimes \alpha$, the elements $\alpha$ and $\beta$ are primitive, i.e. $\Delta(\alpha)=\Delta(\beta)=0$. Since the space $S^n \times S^m$ is formal this is a model for  $S^n \times S^m$.
\end{lemma}

\begin{proposition}\label{propsnxsmcw}
Let $C=H_*(S^n \times S^m)$ be the homology of $S^n \times S^m$. Assume that $n \leq m$. An algebraic CW-decomposition of $C$ is given by first taking one $n$-cell $\alpha$, then we attach an $m$-cell $\beta$ via the zero map. The last $(n+m)$-cell  $\gamma$ is attached via the  $\infty$-morphism $\epsilon:\Omega_{\iota}\gamma \rightarrow \Omega_{\iota}(\R\alpha \oplus \R\beta)$ given by $\epsilon_2(\gamma)=[\alpha,\beta]$ and zero otherwise.
\end{proposition}

\begin{proof}
To define an algebraic CW-complex for $C=H_*(S^n \times S^m)$, we proceed by induction on the skeleta of $H_*(S^n \times S^m)$. First we will start with an $n$-cell $\alpha$ and attach an $m$-cell $\beta$ along the zero map. This gives us the the trivial two dimensional coalgebra $C_{\leq m}$. To obtain $H_*(S^n \times S^m)$, we need to attach an $(n+m)$-cell $\gamma$ in such a way that $\Delta(\gamma)=\alpha \otimes \beta +(-1)^{\mid \alpha \mid \mid \beta \mid}\beta \otimes \alpha$. It is easy to check that the attachment along the map $\epsilon\rightarrow [\alpha,\beta]$ gives us $H_*(S^n \times S^m)$.
\end{proof}

With this algebraic CW-decomposition we can now prove the following theorem about the moduli space of Maurer-Cartan elements. Before we state the theorem we need to fix a basis for $L$. Let $\{\epsilon_{i_k}\}_{i_k \in I_k}$ be a basis for $L_{k}$, the degree $k$-part of $L$. A basis for $Hom_{\R}(H_*(S^n \times S^m),L)_0$ is then given by $\{\varphi_{\alpha,i_n}\}_{i_n\in I_n}\cup \{\varphi_{\beta,j_m}\}_{j_m\in I_m} \cup \{ \varphi_{\gamma,l_{n+m}} \}_{l_{n+m} \in I_{n+m}}$, where $\varphi_{\alpha,i_n}$ is the function given by $\varphi_{\alpha,i_n}(\alpha)=\epsilon_{i_n}$ and zero otherwise. The functions $\varphi_{\beta,j_k}$ and $\varphi_{\gamma,l_{n+m}}$ are defined similarly. The Maurer-Cartan elements $\tau$ and $\kappa$ can be expressed in this basis and we will denote the coefficient of the basis element $\varphi_{\alpha,i_n}$ of $\tau$ by $\lambda_{\alpha,i_n}^{\tau}$. Then we get the following statement about the coefficients $\lambda_{n+m,l}^{\tau}$ and $\lambda_{n+m,l}^{\kappa}$.

\begin{theorem}
Let $\tau,\kappa:Hom_{\R}(H_*(S^n \times S^m) , Y)$ be two Maurer-Cartan elements in the convolution algebra between $H_*(S^n \times S^m)$ and a simply-connected minimal $s^{-1}L_{\infty}$-algebra $Y$ of finite type. The Maurer-Cartan elements $\tau$ and $\kappa$ are gauge equivalent if and only if the following conditions are satisfied:

\begin{enumerate}
 \item The coefficients $\lambda_{\alpha,j_n}^{\tau}$ and $\lambda_{\alpha,j_n}^{\kappa}$ are equal for all $j_n \in I_n$.
 \item The coefficients $\lambda_{\beta,j_m}^{\tau} $ and $\lambda_{\beta,j_m}^{\kappa}$ are equal for all $j_m \in I_m$. We will denote the degree $\leq m$ part of the Maurer-Cartan  element $\kappa$ by $\theta=q_m(\kappa)$, explicitly this is given by $\theta=\sum_{j_n\in I_n} \lambda^{k}_{\alpha,j_n}+\sum_{j_m \in I_m}\lambda^{\kappa}_{\beta,j_m}$.
 \item The coefficients $\lambda_{\gamma,j_{n+m}}^{\tau}$ and $\lambda_{\gamma,j_{n+m}}^{\kappa}$ are equal in the quotient  $(\R\gamma \otimes L_{n+m})/ Im(\nu)$, where  $\nu_{\theta}:\pi_1(Hom_{\R}(C_{\leq m},L))\rightarrow \pi_1(Hom_{\R}(\mathfrak{S}^{n+m},  L))$ is the map given by $\nu_{\theta}(f)(x)=[\theta,x]$, for $x\in Hom_{\R}(C_{\leq n+m-1},L)$.
\end{enumerate}

\end{theorem}

\begin{proof}
To prove the theorem we will use induction on the CW-decompo\-sition of $H_*(S^n \times S^m)$. The algebraic CW-decomposition of $H_*(S^n \times S^m)$ has three cells, so the space $Hom_{\R}(H_*(S^n \times S^m),L)_0$ can be decomposed as $\left(\R\alpha\otimes L_n\right) \oplus \left( \R\beta \otimes L_m \right)\oplus \left(\R\gamma \oplus L_{n+m}\right)$. Since $\alpha$ is the lowest dimensional cell there is no gauge equivalence possible on the space $\R\alpha\otimes L_n$. The elements $q_n(\tau)$ and $q_n(\kappa)$ are therefore invariants of $\tau$ and $\kappa$. 

 The next step is to determine whether $q_m(\tau)$ and $q_m(\kappa)$ are gauge equivalent or not. According to Proposition \ref{proppropprop}, $q_m(\tau)$ and $q_m(\kappa)$ are gauge equivalent if and only if they differ by an element \\ $\gamma \in Im((a^*_{m})^{q_m-1(\tau)})$. Since  $a_m$ is the zero map  and twisting the zero map also gives the zero map, the image of $(a^*_m)^{q_m(\tau)}$ is zero. The Maurer-Cartan elements $q_m(\tau)$ and $q_m(\kappa)$ are therefore gauge equivalent if and only if they are equal. The elements $q_m(\tau)$ and $q_m(\kappa)$ are therefore invariants of the gauge equivalence classes of $\tau$ and $\kappa$.
 
When $q_m(\tau)=q_m(\kappa)$, we need to compare $q_{n+m}(\tau)$ and $q_{n+m}(\kappa)$. To do this we will first denote  $q_m(\tau)$ by $\theta$. We know from Proposition \ref{proppropprop}, that $\pi_1(Hom_{\R}(H_*(S^n \times S^m,L)))$ acts transitively on the fiber over $\theta$. According to Proposition \ref{proppropprop}, two elements are gauge equivalent if and only if they differ by an element of $Im((a_{n+m}^*)^{\theta}_1)$. So all that is left is to compute $Im((a_{n+m}^*)^{\theta}_1)$. After a straightforward application of Definition \ref{definfinitymorphism2}, it follows that the degree one, arity one component of $(a_{n+m}^*)^{\theta}$ is given by the map $\nu$. So the Maurer-Cartan elements $\tau$ and $\kappa$ are gauge equivalent if and only if coefficients $\lambda_{\gamma,j_{n+m}}^{\tau}$ and $\lambda_{\gamma,j_{n+m}}^{\kappa}$ are equal in the quotient $(\R\gamma \otimes L_{n+m})/ Im(\nu)$, which proves the theorem.

\end{proof}

We will now give two examples of a target space $s^{-1}L_{\infty}$-algebra $L$ and the implications this has on  the moduli space of Maurer-Cartan elements.

\begin{example}
When the target $s^{-1}L_{\infty}$-algebra $L$ is abelian, then the products $l_n$ are zero. So in particular the image of $\nu$ is zero. From this it follows that that two Maurer-Cartan element $\tau$ and $\kappa$ are gauge equivalent if and only if  $\lambda_{i,j}^{\tau}=\lambda_{i,j}^{\kappa}$ for all $i$ and $j$.
\end{example}

\begin{example}
We will now describe the moduli space of Maurer-Cartan elements from $S^2 \times S^2$ to $M=S^2 \times S^2 \setminus \{*\}$. The homotopy $s^{-1}L_{\infty}$-algebra of $M$ is the same as the homotopy $s^{-1}L_{\infty}$-algebra $\pi_*(S^2 \vee S^2)$ and is given by $s^{-1}\mathcal{LIE}(x,y)$ the free shifted Lie algebra on two generators $x$ and $y$. The degree $2$ part of the moduli space of Maurer-Cartan elements is now $4$-dimensional and given by $Hom_{\R}(C_2,L_2)$. If we look at the image of the map  $\nu_{\theta}$, then it is straightforward see that $\nu_{\theta}$ is surjective  when $\theta \neq 0$ and $\nu_\theta$ is the zero map when $\theta$ is zero. From this we conclude that the Maurer-Cartan elements $\tau$ and $\kappa$ are gauge equivalent if and only if either 
\begin{enumerate}

\item We have the following equalities  $\lambda_{\alpha,x}^{\tau}=\lambda_{\alpha,x}^{\kappa}$, $\lambda_{\alpha,y}^{\tau}=\lambda_{\alpha,y}^{\kappa}$, $\lambda_{\beta,x}^{\tau}=\lambda_{\beta,x}^{\kappa}$ and $\lambda_{\beta,y}^{\tau}=\lambda_{\beta,y}^{\kappa}$ and at least one of the coefficients is non-zero.
\item Or if all the coefficients $\lambda_{\alpha,x}^{\tau}=\lambda_{\alpha,x}^{\kappa}=\lambda_{\alpha,y}^{\tau}=\lambda_{\alpha,y}^{\kappa}=\lambda_{\beta,x}^{\tau}=\lambda_{\beta,x}^{\kappa}=\lambda_{\beta,y}^{\tau}=\lambda_{\beta,y}^{\kappa}=0$ are zero and $\lambda_{\gamma,i}^{\tau}=\lambda_{\gamma,i}^{\kappa}$ for all basis elements $i \in L_4$. 
\end{enumerate}
\end{example}



When we fix a basis for $L$ we get the following corollary of this theorem. To do this we will use the same notation for the  basis as in the example from Section \ref{secsxs}. Then we get the following statement about the coefficients $\lambda_{n+m,l}^f$ and $\lambda_{n+m,l}^g$.

\begin{corollary}
Let $f,g:S^n \times S^m \rightarrow N$ be two maps  such that $N$ is a simply-connected manifold of finite type. Then the maps $f$ and $g$ are real homotopic if and only if the following conditions are satisfied. 
\begin{enumerate}
\item The coefficients $\lambda_{n,l}^f$ and $\lambda_{n,l}^g$ are equal for all $l \in I_n$.
\item The coefficients $\lambda_{m,l}^f$ and $\lambda_{m,l}^g$ are equal for all $l \in I_m$.
\item The coefficients $\lambda_{n+m,l}^f=\lambda_{n+m,l}^g$ for all $l \in (\R\gamma \otimes L_{n+m})/Im(a_{n+m}^{q_{m}})$.
\end{enumerate}
\end{corollary}

\subsubsection{Detecting the zero Maurer-Cartan element}

In this example we will show how our methods can be used to detect the zero Maurer-Cartan element. In Lemma  \ref{lemtheconstantmapiszero} we show that the zero Maurer-Cartan element corresponds to the constant map. This gives us therefore a method to determine whether a map is real homotopic to the constant map or not. 

\begin{proposition}\label{propdetectingthezeromcelement}
Let $\tau$ be a Maurer-Cartan element in $Hom_{\R}(C,L)$, where we assume that $C$ and $L$ are both minimal, i.e. $d_C=d_L=0$. The Maurer-Cartan element $\tau$ is gauge equivalent to the zero Maurer-Cartan element if and only if it is the zero Maurer-Cartan element.  
\end{proposition}

\begin{proof}
To prove the proposition we will use induction on the algebraic CW-decomposition of $C$. We will show that if $\tau$ is gauge equivalent to $0$, then all the coefficients of $\tau$ are zero. 

We first note that since $Hom_{\R}(C_{2},L)$ is abelian,  there is no gauge equivalence on $Hom_{\R}(C_2,L)$. Since $\tau$ is gauge equivalent to $0$, the element $q_2(\tau)$ has to be equal to zero.

Now assume $q_n(\tau)$ is zero, then we want to show that this implies that $q_{n+1}(\tau)$ is also equal to zero. In other words, we want to show that $q_{n+1}(\tau)$ is gauge equivalent to the zero Maurer-Cartan element if and only if $q_{n+1}(\tau)=0$. To show this we will use Proposition \ref{proppropprop}, which states that $q_{n+1}(\tau)$ and $0$ are gauge equivalent if and only if $\tau$ and $0$ differ by an element of the linear subspace $Im(a_{n+1}^*)$. So to prove the proposition we need to show that $Im(a_{n+1}^*)$ is equal to the zero subspace. Since we assumed that the algebraic CW-complex $C$ is a minimal CW-complex, the linear part of the attaching map is zero. Suppose that the linear part of the attaching map was not equal to zero, then  it would induce a differential on the CW-complex $C$ and therefore making it  no longer minimal. Since the linear part of the attaching map is zero, the image of the linear part is also zero, i.e. $Im(a_{n+1}^*)=\{0\}$. Now we can apply Proposition \ref{proppropprop}, which implies that, $q_{n+1}(\tau)$ and $0$ are gauge equivalent if and only if they differ by an element of the zero subspace, so they have to be equal. 

So because of this induction step, all the elements $q_n(\tau)$ have to be equal to zero for $\tau$ to be gauge equivalent to $0$. Since all the $q_n(\tau)$ have to be zero the Maurer-Cartan element $\tau$ has to be zero.
\end{proof}

One application of Theorem \ref{propdetectingthezeromcelement} is that it gives us a way to determine if a map is real homotopic to the constant map or not. To do this we first need to show that the Maurer-Cartan element of the constant map is $0$. Then we can apply Theorem \ref{propdetectingthezeromcelement} to show that a map $f:M \rightarrow N$ is real homotopic to the zero map if and only if all the coefficients $\lambda_{i,j}^f$ vanish.

\begin{lemma}\label{lemtheconstantmapiszero}
 Let $c:M \rightarrow N$ be the constant map between two simply-connected manifolds $M$ and $N$ such that $M$ is compact. The Maurer-Cartan element corresponding to the constant map is the zero Maurer-Cartan element, i.e. the coefficients $\lambda_{i,j}^{c}$ are all equal to zero.
\end{lemma}

\begin{proof}
To compute the coefficients of the constant map we notice that the pullback of a differential $k$-form $\omega$ under the constant map is always equal to zero if $k \geq 1$. Since the explicit formulas for the  coefficients $\lambda_{i,j}^c$ all involve the pull back of  differential forms of degree greater than $1$ all these formulas will vanish. The Maurer-Cartan element of the constant map is therefore equal to zero.
\end{proof}

\begin{theorem}\label{thrmzeromap}
 Let $f:M \rightarrow N$ be a map between two simply-connected manifolds $M$ and $N$, such that $M$ is compact. The map $f$ is real homotopic to the constant map if and only if all the coefficients $\lambda_{i,j}^f$ of $mc(f)$ are equal to zero. This equivalent to the vanishing of all the integrals from Lemma \ref{lemconcreteformulas}. 
\end{theorem}

\begin{proof}
This is a straightforward consequence of Proposition \ref{propdetectingthezeromcelement} and  Lemma \ref{lemtheconstantmapiszero}. Because of Lemma \ref{lemtheconstantmapiszero}, the constant map has the zero Maurer-Cartan element. Because of Proposition \ref{propdetectingthezeromcelement}, a Maurer-Cartan element is gauge equivalent to the zero Maurer-Cartan element if and only if it is the zero Maurer-Cartan element. A map $f:M \rightarrow N$ is therefore real homotopic to the zero Maurer-Cartan element if and only if $mc_{\infty}(f)=0$. This is the same as all the coefficients $\lambda_{i,j}^{f}$ being equal to zero.
\end{proof}

\part{Examples and other approaches}\label{partexample}

In the last part of this paper we will demonstrate how our Hopf invariants work in practice, by working out some concrete examples. For our examples it will be necessary to construct the morphisms $q$ and $j$ from Definition \ref{defalghopfinv}. We will first give a few ad hoc examples in which we do not need to use the Homotopy Transfer Theorem to define the map $j$. 

After these examples we will explain how we can use Hodge Theory to construct the map $j$. For the construction of the $\infty$-morphism $j$ from Sections \ref{secalgebraichopfinv} and \ref{secfromspacestomc} we want to apply the Homotopy Transfer Theorem. To do this, we need explicit formulas for a contraction as in Definition \ref{defcontraction2}. For compact oriented Riemannian manifolds without boundary, Hodge Theory provides us with such a contraction for the de Rham complex. In Section \ref{secHodgetheory},  we will therefore recall some facts from the Hodge theory of Riemannian manifolds and then show how to apply this to maps from compact oriented Riemannian manifolds without boundary to other (not necessarily compact oriented Riemannian) manifolds. 




\section{Examples}\label{secexamples}

We will now give a few examples to show that the invariants defined in this paper are in practice often very computable. We will give three maps from $S^2 \times S^2$ to another space and show that two of them are not homotopic to the constant map and the third map is real homotopic to the constant map.

To do this we will need a concrete description of $S^2 \times S^2$. We will therefore describe $S^2 \times S^2$ as the subset of $\R^6$ given by 
$$S^2 \times S^2=\{(x,y,z,u,v,w)\in \R^6 \mid x^2+y^2+z^2=1 \mbox{ and } u^2+v^2+w^2=1\} .$$

The homology of $S^2 \times S^2$ is given by $H_*(S^2\times S^2)=\R \alpha \oplus \R \beta \oplus \R \gamma$, where $\alpha$ and $\beta$ have degree $2$ and the degree of $\gamma$ is $4$. The coproduct $\Delta$ is given by $\Delta(\alpha)=\Delta(\beta)=0$ and $\Delta(\gamma)=\alpha \otimes \beta +\beta \otimes \alpha$. A representative for the homology class $\alpha$ is given by the submanifold given by $\{(x,y,z,1,0,0) \in \R^6 \mid x^2+y^2+z^2=1\}$ and a representative for the class $\beta $ is given by the submanifold $\{(1,0,0,u,v,w) \in \R^6 \mid u^2+v^2 +w^2=1 \}$. The class $\gamma$ is represented by $S^2 \times S^2$.

In the next two examples we will compute the Maurer-Cartan element of maps from $S^2 \times S^2 $ to $\R^3\setminus (0,0,0)$. Note that $\R^3\setminus (0,0,0)$ is homotopy equivalent to $S^2$ and has as Lie model $\mathcal{LIE}(\xi)$, the free Lie algebra on a generator $\xi$ of degree $2$. The space of Maurer-Cartan elements is in this case given by $Hom_{\R}(H_2(S^2 \times S^2),\pi_2(\R^3 \setminus (0,0,0)))$. It is easy to check that there is no gauge equivalence on this set. So the moduli space of Maurer-Cartan elements is isomorphic to $\R^2$ and is spanned by $\varphi_{\alpha,\xi}$ and $\varphi_{\beta,\xi}$, where the element $\varphi_{\alpha,\xi}$ is the basis element of $Hom_{\R}(H_2(S^2 \times S^2),\pi_2(\R^3 \setminus (0,0,0)))$ given by $\varphi_{\alpha,\xi}(\alpha)=\xi$ and $\varphi_{\alpha,\xi}(\beta)=0$ and $\varphi_{\beta,\xi}$ is defined similarly.

So in general, the real homotopy class of a map $f:S^2 \times S^2 \rightarrow \R^3 \setminus (0,0,0)$ is defined by a linear combination $\lambda^f_{\alpha,\xi} \varphi_{\alpha,\xi}+\lambda^f_{\beta,\xi} \varphi_{\beta,\xi}$. We would like to compute the coefficients of a map $f$ using Theorem \ref{thrmhowtocomputeMaurarCartanelements}. To do this we first need to fix the $\infty$-morphism $j:\Omega^{\bullet}(S^2\times S^2) \rightsquigarrow H^*(S^2 \times S^2)$ and the map $q:\pi^*(\R ^3 \setminus (0,0,0)) \rightarrow B_{s^{-1}Lie}\Omega^{\bullet}(\R^3 \setminus (0,0,0))$ from Section \ref{secalgebraichopfinv}. 

To define the $\infty$-morphism $j$ we will use the Homotopy Transfer Theorem from Section \ref{secHTT}. So we need to define an inclusion $i$ and projection $p$ of $H^*(S^2 \times S^2)$ to $\Omega^{\bullet}(S^2 \times S^2)$ and a contraction $H:\Omega^{\bullet}(S^2 \times S^2) \rightarrow \Omega^{\bullet}(S^2 \times S^2)$. We will first define $i$. We will do this by representing $H^2(S^2 \times S^2)$ by the following two differential forms
$$\omega=\frac{1}{2 \pi} \frac{x dy \wedge dz-y dx \wedge dz+ z dx \wedge dz}{x^2+y^2 +z^2},$$
$$\psi=\frac{1}{2 \pi} \frac{u dv \wedge dw-v du \wedge dw+ w du \wedge dv}{u^2+v^2 +w^2}.$$
The cohomology group $H^4(S^2 \times S^2)$ will be represented by $\omega \wedge \psi$. The inclusion map $i$ is defined by sending $\omega$, $\psi$ and $\omega \wedge \psi$ to the corresponding differential forms in $\Omega^{\bullet}(S^2 \times S^2)$. To define the map $p$ we fix an inner product on $\Omega^{\bullet}(S^2 \times S^2)$ and project onto the subspace spanned by $\omega$ ,$\psi$ and $\omega \wedge \psi$. The map $H$ will not be important for the example and it is enough to know it exists.

The map $q:\pi^*(\R^3 \setminus (0,0,0)) \rightarrow B_{s^{-1}Lie}\Omega^{\bullet}(\R^3 \setminus (0,0,0))$ will be defined as follows. The cohomotopy groups $\pi^*(\R^3 \setminus (0,0,0))$ are given by the free shifted Lie coalgebra on one gener ator $\zeta$ and has as basis $\zeta$ and $]\zeta,\zeta[$, where $]\zeta, \zeta[$ is the cobracket of $\zeta$ with itself. Then we define the differential form 
$$\Upsilon=\frac{1}{2 \pi} \frac{x dy \wedge dz-y dx \wedge dz+ z dx \wedge dz}{x^2+y^2 +z^2}$$
on $\R^3 \setminus (0,0,0)$. Then we define $q$ by $q(\zeta)=\Upsilon$ and $q(]\zeta,\zeta[)=]\Upsilon,\Upsilon[$.

\begin{remark}
For the examples we will consider it is not necessary to know the explicit contraction $H:\Omega^{\bullet}(S^2 \times S^2) \rightarrow \Omega^{\bullet}(S^2 \times S^2)$. In Section \ref{secHodgetheory} we will give formulas for the more general case.
\end{remark}

\begin{example}\label{exintegral1}
Let $p_1:S^2 \times S^2 \rightarrow \R^3 \setminus (0,0,0)$ be the projection onto the subspace of $\R^6$ spanned by $x$, $y$ and $z$ with the origin removed. In coordinates this map is given by $p_1(x,y,z,u,v,w)=(x,y,z)$. 

We would now like to compute the coefficients of the map $p_1$. According to Theorem \ref{thrmhowtocomputeMaurarCartanelements} these are given by the formulas given in that theorem, since we picked the maps $j$ and $q$ in such a way the coefficients become:
$$\lambda^{p_1}_{\alpha,\xi}= \int_{\alpha} p_1^* \Upsilon,  $$
$$ \lambda^{p_1}_{\beta,\xi}= \int_{\beta} p_1^* \Upsilon.  $$

Since $p_1^* \Upsilon$ is equal to $\omega$ we get the integrals
$$\lambda^{p_1}_{\alpha,\xi}= \int_{\alpha} \omega,  $$
$$ \lambda^{p_1}_{\beta,\xi}= \int_{\beta} \omega  .$$
If we use spherical coordinates we can compute the values of these integrals. The coefficients are given by $\lambda_{\alpha,\xi}^{p_1}=1$ and $ \lambda^{p_1}_{\beta,\xi}=0$. 

\end{example}

\begin{example}
Let $p_2:S^2 \times S^2 \rightarrow \R^3 \setminus (0,0,0)$ be the projection onto the subspace of $\R^6$ spanned by $u$, $v$ and $w$ with the origin removed. In coordinates this map is given by $p_2(x,y,z,u,v,w)=(u,v,w)$.

We would now like to compute the coefficients of the map $p_2$. Similar to Example \ref{exintegral1} we can use the formulas from Theorem \ref{thrmhowtocomputeMaurarCartanelements} and get the following integrals:
$$\lambda^{p_2}_{\alpha,\xi}= \int_{\alpha} p_2^* \Upsilon,  $$
$$ \lambda^{p_2}_{\beta,\xi}= \int_{\beta} p_2^* \Upsilon.  $$
Since $p_2 \Upsilon=\psi$ we get the following integrals
$$\lambda^{p_2}_{\alpha,\xi}= \int_{\alpha} \psi,  $$
$$ \lambda^{p_2}_{\beta,\xi}= \int_{\beta} \psi  .$$
Again by using spherical coordinates we find that the values of these integrals are given by $\lambda^{p_2}_{\alpha,\xi}=0$ and $ \lambda^{p_2}_{\beta,\xi}=1$.
\end{example}

\begin{corollary}
Since the coefficients $\lambda_{\alpha,\xi}^{p_1}$ and $\lambda_{\alpha,\xi}^{p_2}$ don't agree, the maps $p_1$ and $p_2$ are not homotopic.
\end{corollary}

The last example we give is  the inclusion of $S^2 \times S^2$ into $Y$, where the space $Y$ is defined as $\R^6\setminus \{(0,0,0,0,0,\lambda) \mid \lambda \in \R\}$. The space $Y$ is homotopy equivalent to $S^4$, this can be seen by first retracting $Y$ along the $w$-axis to $\R^5 \setminus (0,0,0,0,0)$, which is clearly homotopy equivalent to $S^4$. The map $i:S^2 \times S^2 \rightarrow Y$ is then the inclusion map, given by $i(x,y,z,u,v,w)=(x,y,z,u,v,w)$. 

A representative for $\pi^4(Y)$ is given by 
$$\zeta=x dy \wedge dz \wedge du \wedge dv - y dx \wedge dz \wedge du \wedge dv+z dx \wedge dy \wedge du \wedge dv$$
$$  - u dx \wedge dy \wedge dz \wedge dv +v  dx \wedge dy \wedge dz \wedge du .$$
In this case the degree zero part of $Hom_{\R}(H_*(S^2 \times S^2),\pi_*(Y))$ is given by $Hom_{\R}(H_4(S^2 \times S^2),\pi_4(Y))$, which is isomorphic to $\R$. It is straightforward to check that every element in $Hom_{\R}(H_4(S^2 \times S^2),\pi_4(Y))$ is a Maurer-Cartan element and that there is no gauge equivalence. So the real homotopy class of a map $f:S^2 \times S^2 \rightarrow Y$ is determined by a real number. So in particular $mc(f)=mc_{\infty}(f)$.

To compute the Maurer-Cartan element corresponding to $i:S^2 \times S^2 \rightarrow Y$ we need to integrate the pull back of $\zeta$ over $S^2 \times S^2$. Since $i$ is the inclusion map, the pull back $i^*\zeta$ is  the  restriction of $\zeta$ to $S^2 \times S^2$. So we have to solve the integral
$$mc(i)=\int_{S^2 \times S^2 }i^*\zeta .$$
By passing to spherical coordinates in the first three variables and spherical coordinates in the last three variables we can compute the integral. After a fairly long computation we find
$$\int_{S^2 \times S^2 }i^*\zeta=0 .$$
  
\begin{corollary}
The map $i:S^2 \times S^2 \rightarrow Y$ is real homotopic to the constant map.
\end{corollary}  

\begin{proof}
Since the Maurer-Cartan element $mc(i)$ is equal to zero we can apply Theorem \ref{thrmzeromap} which states that $i$ is real homotopic to the constant map if and only if $mc(i)=0$.  
\end{proof}

\section{Some facts from Hodge theory and an explicit contraction}\label{secHodgetheory}

To work out more complicated examples than in Section \ref{secexamples} we need general formulas for the map $j$ from Definition \ref{defalghopfinv}. To construct these formulas we will use Hodge theory to construct an explicit contraction when $M$ is a compact oriented Riemannian manifold without boundary. The Hodge decomposition of the de Rham complex $\Omega^{\bullet}(M)$ provides us with explicit choices for a contraction of the de Rham complex as in Definition \ref{defcontraction2}. Most of this section is based on Section 3.5 and Appendix A of \cite{FOT1}. 


Let $M$ be a compact oriented Riemannian manifold without boundary and $\Omega^{\bullet}(M)$ the de Rham complex of differential forms. The metric of the manifold $M$ induces an inner product of the space of differential forms which is defined by 
$$\left<\alpha,\beta\right>=\int_M \alpha \wedge *\beta ,$$
for two differential forms $\alpha,\beta \in \Omega^q(M)$, the operator $*$ is here the Hodge star operator. The inner product is defined to be zero if the degrees of $\alpha$ and $\beta$ differ. Using this inner product, the codifferential $\delta:\Omega^{\bullet}(M) \rightarrow \Omega^{\bullet-1}(M)$ is defined as the adjoint of the exterior derivative. The Laplace operator is defined as $\Delta=\delta d + d \delta$ and a differential form $\omega$ is called harmonic if $\Delta \omega =0$. The space of harmonic forms of degree $p$, is denoted by $\mathcal{H}^{p}(M)$ and will be important since it represents the cohomology of $M$, which is summarized in the following theorem.

\begin{theorem}\label{thmhodgedecomp}
 The space of degree $p$ de Rham forms $\Omega^p(M)$ has a decomposition called the Hodge decomposition which is given by
 $$\Omega^p(M)=\mathcal{H}^p(M) \oplus Im(d) \oplus Im(\delta).$$ 
 In particular there is an isomorphism of vector spaces between $\mathcal{H}^p(M)$ and $H^p(M)$ and every differential form $\alpha \in \Omega^p(M)$ has a unique decomposition
 $$\alpha=\mathcal{H}(\alpha)+\alpha_d+\alpha_{\delta},$$
 where $\mathcal{H}(\alpha)\in \mathcal{H}^p(M)$, $\alpha_d \in Im(d)$ and $\alpha_{\delta} \in Im(\delta)$.
\end{theorem}

It follows from this theorem that the restriction of the differential $d$ gives an isomorphism between $Im(\delta)$ and $Im(d)$. This is summarized in Lemma A.11 of \cite{FOT1}.

\begin{lemma}\label{lem9.2}
 If $\alpha\in Im(d)$ is a differential form then there exists a unique  $\beta \in Im(\delta)$ such that $d(\beta)=\alpha$.  
\end{lemma}

Because of Lemma \ref{lem9.2} and Theorem \ref{thmhodgedecomp}, the subcomplex $Im(\delta) \oplus Im(d)$ forms an acyclic subcomplex of $\Omega^{\bullet}(M)$. The restriction of the differential induces therefore an isomorphism $d:Im(\delta) \rightarrow Im(d)$.  We denote the inverse of $d$ by $Q:Im(d) \rightarrow Im(\delta)$. Using $Q$ we can now define a contraction of the de Rham complex.

\begin{definition}\label{defcontraction} 
 The operator $H:\Omega^{\bullet}(M) \rightarrow \Omega^{\bullet-1}(M)$ is defined as the linear map of degree $-1$, which is zero on $\mathcal{H}^\bullet(M)$ and $Im(\delta)$ and equal to $Q$ on $Im(d)$. We will call this contraction the Hodge contraction.
\end{definition}

\begin{lemma}
Let $i:\mathcal{H}^{\bullet}(M)\rightarrow \Omega^{\bullet}(M)$ be the inclusion of the harmonic forms of $M$ into the de Rham complex by sending a cohomology class $\alpha$ to the harmonic form representing $\alpha$. Let $p:\Omega^{\bullet}(M)\rightarrow \mathcal{H}^{\bullet}(M)$ be the projection onto the harmonic forms induced by the inner product $<,>$ and let $H$ be the operator from Definition \ref{defcontraction}. The map $H$ is then a contraction of chain complexes as in Theorem \ref{thrmHTT}.
\end{lemma}

\begin{proof}
 To show that $H:\Omega^{\bullet}(M) \rightarrow \Omega^{\bullet}(M)$ is a contraction we  need to show that $Id_{\Omega^{\bullet}(M)}-ip=dH+Hd$ and that $H$ satisfies the other conditions given by $HH=0$, $pH=0$ and $Hi=0$. This is a straightforward check and is left to the reader.
\end{proof}

The Hodge contraction gives us an explicit choice for the contraction in Lemma \ref{lemconcreteformulas} and therefore gives us explicit formulas for the higher Hopf invariants.

\section{Rational homotopy theory and alternative approach\-es}\label{secalternativeapproaches}

In the last section of this paper, we give a sketch how the ideas of this paper can be generalized to rational homotopy theory by using Sullivan's polynomial de Rham forms instead of the de Rham complex. We will also give a short comparison of our approach with the paper "Obstructions to homotopy equivalences" by Halperin and Stasheff (see \cite{HS1}). This paper is similar in spirit to the ideas of this paper and we will briefly explain the main differences between our questions and approaches.

\subsection{Rational homotopy theory}

All the ideas previously described in this paper can also be used to determine whether two maps are rationally homotopic. In this section we will briefly describe how the ideas developed above, can also be applied to Sullivan's polynomial de Rham forms. The main disadvantages of working with polynomial de Rham forms instead of the smooth differential forms, are that there is no explicit choice of contraction for polynomial de Rham forms and that the theory of integration of smooth differential forms is a lot more advanced than the theory of integration of polynomial de Rham forms. The main advantage of working with polynomial de Rham forms, is that the theory is now applicable to more general spaces than just manifolds.

Since all the proofs of the statements in this section are completely analogous to the proofs of the corresponding statements in the real case, we will omit the proofs.

\begin{convention}
In this section we will assume that $X$ is a $1$-reduced simplicial set with finite dimensional rational homology coalgebra and that $Y$ is a simply-connected rational simplicial set of finite $\Q$-type and that $Y_{\Q}$ is the rationalization of $Y$. We will also assume that all homology and homotopy groups are taken with rational coefficients.
\end{convention}

Similar to the manifold case we can define maps $mc:Map_*(X,Y_{\Q})\rightarrow Hom_{\Q}(H_*(X),\pi_*(Y))$ and \\ $mc_{\infty}:Map_*(X,Y_{\Q})\rightarrow \mathcal{MC}(Hom_{\Q}(H_*(X),\pi_*(Y))$. These maps are defined as in Section \ref{secalgebraichopfinv}, the only difference is now that instead of using the de Rham complex we will use Sullivan's complex of polynomial de Rham forms. If we fix a basis $\{\varphi_{i,j} \}$ for $Hom_{\Q}(H_*(X),\pi_*(Y))$ as in Section \ref{secfromspacestomc}, then we can compute the coefficients of a map $f:X \rightarrow Y_{\Q}$ as described in the following theorem.

\begin{theorem}\label{thrmrationalanalog}
Let $f:X \rightarrow Y$ be a map between a simply-connected simplicial set $X$ with finite dimensional rational cohomology ring and let $Y$ be a simply-connected space of finite $\Q$-type. Let $\{\alpha_i\}$ be a set of subspaces of $X$ representing $H_*(X)$.  The coefficients of the map $mc(f)$ can be computed as 
$$\lambda_{i,j}^f=\sum_{n \geq 1} \int_{\alpha_i} \left((-1)^{\mid \nu \mid} p t^{\nu} \mathbf{h}_n(f^*q_n(\omega_j)) + \sum_{k=1}^{u} (-1)^{\mid \nu_k '' \mid}(p^{\nu_k '} \circ_{e_k} t^{ \nu_k ''}) \tau_k \mathbf{h}_n (f^*q_n(\omega_j)) \right)$$
\end{theorem}

\begin{corollary}
 Under the hypotheses of Theorem \ref{thrmrationalanalog}, suppose that we have two maps $f,g:X \rightarrow Y$. The maps $f$ and $g$ are rationally homotopic if and only if the Maurer-Cartan elements $mc(f)$ and $mc(g)$ are gauge equivalent in $Hom_{\Q}(H_*(X),\pi_*(Y_{\Q}))$.
\end{corollary}


Similar to the real case we can use the algebraic CW-complexes to get statements about the coefficients. An example of such a statement is the following theorem.

\begin{theorem}
 Let $f:X \rightarrow Y$ be a map between a simply-connected space $X$ with finite dimensional cohomology ring and $Y$,  a simply-connected  space of finite $\Q$-type. The map $f$ is rationally homotopic to the constant map if and only if all the coefficients $\lambda^f_{i,j }$ are equal to zero.
\end{theorem}

\subsection{Comparison with "Obstructions to homotopy equivalences"}

We will now give a brief comparison of our paper with \cite{HS1}. The paper \cite{HS1} is similar in spirit to this paper, but in some sense their main question is the complete opposite of the question that we try to answer.

The main topic of \cite{HS1} is the following question:

\begin{question}\label{Quest1}
Let $X$ and $Y$ be simply connected spaces and let $f:H^*(Y) \rightarrow H^*(X)$ be an isomorphism between rational cohomology rings, can $f$ be realized as a rational homotopy equivalence between $X$ and $Y$?
\end{question} 

Their approach to solving this problem is to first observe that the answer of this question is equivalent to constructing a map $\tilde{f}:M_Y \rightarrow M_X$ which induces $f$ in cohomology, the algebras $M_X$ and $M_Y$ are here the minimal models of $X$ and $Y$. Their solution is then to construct an obstruction theory to decide whether a map $\tilde{f}$ exists or not. 

The main question of this paper can be reformulated as the following question.

\begin{question}
Given a map $g:X \rightarrow Y$, what are the Maurer-Cartan elements $mc(g)$ and $mc_{\infty}(g)$? 
\end{question} 

We can reformulate Question \ref{Quest1} then as follows: Given an isomorphism $f:H^*(Y) \rightarrow H^*(X)$ can we find a Maurer-Cartan element $\hat{f}\in Hom_{\Q}(\pi_*(Y),M_X)$ such that $\hat{f}$ induces $f$ in cohomology? 

This question is in some sense the complete opposite of what we consider in this paper. Since we assume that we have a given map $g:X \rightarrow Y$ and try to find its Maurer-Cartan elements $mc(g)$ and $mc_{\infty}(g)$. A possible direction of future research is to see if it is possible to put the obstruction theory of \cite{HS1} in the Maurer-Cartan framework of this paper.

\bibliographystyle{plain}

\bibliography{bibliography}{}

\def\cprime{$'$}
\begin{thebibliography}{10}

\bibitem{BL1}
H.~J. Baues and J.-M. Lemaire.
\newblock Minimal models in homotopy theory.
\newblock {\em Math. Ann.}, 225(3):219--242, 1977.

\bibitem{Berg2}
Alexander Berglund.
\newblock Homological perturbation theory for algebras over operads.
\newblock {\em Algebr. Geom. Topol.}, 14(5):2511--2548, 2014.

\bibitem{Berg1}
Alexander Berglund.
\newblock Rational homotopy theory of mapping spaces via {L}ie theory for
  {$L_\infty$}-algebras.
\newblock {\em Homology Homotopy Appl.}, 17(2):343--369, 2015.

\bibitem{CR1}
Joana Cirici and Agust\'i Roig.
\newblock Sullivan minimal models of operad algebras.
\newblock {\em arXiv:1612.03862v1}.

\bibitem{DGMS1}
Pierre Deligne, Phillip Griffiths, John Morgan, and Dennis Sullivan.
\newblock Real homotopy theory of {K}\"ahler manifolds.
\newblock {\em Invent. Math.}, 29(3):245--274, 1975.

\bibitem{FHT}
Yves F{\'e}lix, Stephen Halperin, and Jean-Claude Thomas.
\newblock {\em Rational homotopy theory}, volume 205 of {\em Graduate Texts in
  Mathematics}.
\newblock Springer-Verlag, New York, 2001.

\bibitem{FOT1}
Yves F\'elix, John Oprea, and Daniel Tanr\'e.
\newblock {\em Algebraic models in geometry}, volume~17 of {\em Oxford Graduate
  Texts in Mathematics}.
\newblock Oxford University Press, Oxford, 2008.

\bibitem{Getz1}
Ezra Getzler.
\newblock Lie theory for nilpotent {$L_\infty$}-algebras.
\newblock {\em Ann. of Math. (2)}, 170(1):271--301, 2009.

\bibitem{HS1}
Stephen Halperin and James Stasheff.
\newblock Obstructions to homotopy equivalences.
\newblock {\em Adv. in Math.}, 32(3):233--279, 1979.

\bibitem{Hin1}
Vladimir Hinich.
\newblock Homological algebra of homotopy algebras.
\newblock {\em Comm. Algebra}, 25(10):3291--3323, 1997.

\bibitem{Hin2}
Vladimir Hinich.
\newblock D{G} coalgebras as formal stacks.
\newblock {\em J. Pure Appl. Algebra}, 162(2-3):209--250, 2001.

\bibitem{LV}
Jean-Louis Loday and Bruno Vallette.
\newblock {\em Algebraic operads}, volume 346 of {\em Grundlehren der
  Mathematischen Wissenschaften [Fundamental Principles of Mathematical
  Sciences]}.
\newblock Springer, Heidelberg, 2012.

\bibitem{MP11}
J.~P. May and K.~Ponto.
\newblock {\em More concise algebraic topology}.
\newblock Chicago Lectures in Mathematics. University of Chicago Press,
  Chicago, IL, 2012.
\newblock Localization, completion, and model categories.

\bibitem{Quil1}
Daniel Quillen.
\newblock Rational homotopy theory.
\newblock {\em Ann. of Math. (2)}, 90:205--295, 1969.

\bibitem{RN1}
Daniel Robert-Nicoud.
\newblock Deformation theory with homotopy algebra structures on tensor
  products.
\newblock {\em Doc. Math.}, 23:189--240, 2018.

\bibitem{RNW1}
Daniel Robert-Nicoud and Felix Wierstra.
\newblock Homotopy morphisms between convolution homotopy {L}ie algebras.
\newblock {\em To appear in the Journal of Noncommutative Geometry,
  arXiv:1712.00794}.

\bibitem{SW2}
Dev Sinha and Ben Walter.
\newblock Lie coalgebras and rational homotopy theory {II}: {H}opf invariants.
\newblock {\em Trans. Amer. Math. Soc.}, 365(2):861--883, 2013.

\bibitem{SS2}
Jim Stasheff and Mike Schlessinger.
\newblock Deformation theory and rational homotopy type.
\newblock {\em arXiv:1211.16472}.

\bibitem{Val1}
Bruno Vallette.
\newblock Homotopy theory of homotopy algebras.
\newblock {\em arXiv:1411.5533v3}.

\bibitem{Wie1}
Felix Wierstra.
\newblock Algebraic {H}opf invariants.
\newblock {\em arXiv:1612.07762}.

\end{thebibliography}

\end{document}